\documentclass[preprint,3p]{elsarticle}
\usepackage{graphicx}        
\usepackage{amssymb}
\usepackage{amsmath}
\usepackage{amsthm}
\usepackage{url}
\usepackage{color}

\newtheorem{Theorem}{Theorem}[section]{\bfseries}{\itshape}
\newtheorem{theorem}[Theorem]{Theorem}{\bfseries}{\itshape}
{\bfseries}{\itshape}
\newtheorem{lemma}[Theorem]{Lemma}{\bfseries}{\itshape}
{\bfseries}{\itshape}
{\bfseries}{\itshape}
{\bfseries}{\itshape}
{\bfseries}{\itshape}
\newtheorem{remark}[Theorem]{Remark}{\bfseries}{\itshape}
{\bfseries}{\itshape}
{\bfseries}{\itshape}
\newtheorem{assumption}[Theorem]{Assumption}{\bfseries}{\itshape}
\newtheorem{definition}[Theorem]{Definition}{\bfseries}{\itshape}
\newtheorem{algorithm}[Theorem]{Algorithm}{\bfseries}{\itshape}
{\bfseries}{\itshape}
\newcommand\bi{{\mathbf i}}
\newcommand\BbbR{{\mathbb R}}
\newcommand\BbbC{{\mathbb C}}
\newcommand\BbbN{{\mathbb N}}
\newcommand{\Idx}{\mathcal{I}}
\newcommand{\Idy}{\mathcal{J}}
\newcommand{\sons}{\operatorname*{sons}}
\newcommand{\father}{\operatorname*{father}}
\newcommand{\descendants}{\operatorname*{descendants}}
\newcommand{\level}{\operatorname*{level}}
\newcommand{\depth}{\operatorname*{depth}}
\newcommand{\diam}{\operatorname*{diam}\nolimits}
\newcommand{\dist}{\operatorname*{dist}\nolimits}

\newcommand{\lfiII}{\mathcal{L}_{\Idx\times\Idx}^-}
\newcommand{\lfsaIJ}{\mathcal{L}_{\Idx\times\Idy}^{standard,+}}
\newcommand{\lfsiIJ}{\mathcal{L}_{\Idx\times\Idy}^{standard,-}}
\newcommand{\lfsaII}{\mathcal{L}_{\Idx\times\Idx}^{standard,+}}
\newcommand{\lfsiII}{\mathcal{L}_{\Idx\times\Idx}^{standard,-}}
\newcommand{\lfpa}{\mathcal{L}^{parabolic,+}}
\newcommand{\lfIJ}{\mathcal{L}_{\Idx\times\Idy}}
\newcommand{\lfaIJ}{\mathcal{L}_{\Idx\times\Idy}^+}
\newcommand{\lfiIJ}{\mathcal{L}_{\Idx\times\Idy}^-}
\newcommand{\ellipse}{\mathcal{E}}

\newcommand{\eex}{\hbox{}\hfill\rule{0.8ex}{0.8ex}}
\newcommand{\eremk}{\eex}

\newcommand{\amp}{A}
\newcommand{\phase}{\Phi}
\newcommand{\bs}[1]{\boldsymbol #1}
\newcommand{\punkt}{\boldsymbol{\cdot}}
\newcommand{\enorm}{\mathfrak{n}}

\numberwithin{equation}{section}
\newenvironment{numberedproof}[1]{\textbf{Proof {#1}:}}{\mbox{}\hfill{\hbox{$\Box$}}\par\bigskip}
\newif\iftechreport
\techreporttrue
\iftechreport
\else
\fi

\begin{document}
\begin{frontmatter}

\iftechreport 
\title{An analysis of a butterfly algorithm}
\else
\title{An analysis of a butterfly algorithm}
\fi
 \author[add1]{S. B\"orm} 
\ead[add1]{sb@informatik.uni-kiel.de}
 \author[add1]{C. B\"orst} 
\ead[add1]{chb@informatik.uni-kiel.de}
 \author[add2]{J.M. Melenk}
\ead[add2]{melenk@tuwien.ac.at}
\address[add1]{Institut f\"ur Informatik, Universit\"at Kiel, Christian-Albrechts-Platz 4, D-24118 Kiel}
\address[add2]{Technische Universit\"at Wien, Wiedner Hauptstra\ss e 8-10, A-1040 Wien}


\begin{abstract}
Butterfly algorithms are an effective multilevel technique to compress discretizations of integral operators
with highly oscillatory kernel functions. The particular version of the butterfly algorithm
presented in \cite{CADEYI09} realizes the transfer between levels by Chebyshev interpolation.
We present a refinement of the analysis given in \cite{demanet-ferrara-maxwell-poulson-ying10} 
for this particular algorithm.  
\end{abstract}
\end{frontmatter}
%
\section{Introduction}


Nonlocal operators with highly oscillatory kernel functions  arise in many applications. Prominent examples of such operators 
include the Fourier transform, special function transforms, and integral operators whose kernels are connected 
to the high-frequency Helm\-holtz or Maxwell's equations. Many algorithms with log-linear complexity 
have been proposed in the past. A very incomplete list includes: for the non-uniform Fourier transform 
\cite{keiner-kunis-potts09} (and references therein), \cite{ayliner-chew-song01,CADEYI09,kunis-melzer12}; 
for special function transforms \cite{oneil-woolfe-rokhlin10}; 
for Helmholtz and Maxwell integral operators
\cite{BR91,RO93,BOMI96,GRHUROWA98,DA00,chen-et-al06,ENYI07,MESCDA12,BEKUVE15,boerm-melenk15}.  

Underlying the log-linear complexity of the above algorithms for the treatment of high-frequency kernels is 
the use of a multilevel structure and in particular efficient transfer operators between the levels. In the language
of linear algebra and with $L$ denoting the number of levels,  the matrix realizing the operator is (approximately)
factored into $O(L)$ matrices, each of which can be treated with nearly optimal complexity 
$O(N \log^\alpha N)$, where $N$ denotes the problem size. This observation is formulated explicitly 
in \cite{LIYAYI15} in connection with a class of butterfly algorithms. 

A central issue for a complete error analysis of such algorithms is that of stability. That is, since 
the factorization into $O(L)$ factors results from an approximation process, the application of each of 
the $O(L)$ steps incurs an error whose propagation in the course of a realization of the matrix-vector 
multiplications needs to be controlled. 
For some algorithms, a full error analysis is possible. Here, we mention the fast multipole method \cite{RO93} 
(and its stable realizations \cite{DA00,GRHUROWA98}) for the Helmholtz kernel, which exploits 
that a specific kernel is analyzed. 
One tool to construct suitable factorizations for more general 
high-frequency kernels is polynomial interpolation (typically tensor product Chebyshev interpolation) as proposed, 
e.g., in \cite{kunis-melzer12,CADEYI09} for Fourier transforms and generalizations of the Fourier transform 
and in \cite{MESCDA12,boerm-melenk15} for Helmholtz kernels. 
A full analysis of the case of the (classical) 
Fourier transform is developed in \cite{kunis-melzer12} and for the algorithm proposed in \cite[Sec.~4]{CADEYI09} in 
\cite{demanet-ferrara-maxwell-poulson-ying10}.  For the procedure proposed in \cite{MESCDA12} 
for the Helmholtz kernel, a detailed analysis is given in \cite{boerm-melenk15}. 
Based on generalizations of the tools developed in the latter work \cite{boerm-melenk15} the novel contribution
of the present work is a sharper analysis of the butterfly algorithm proposed in \cite[Sec.~4]{CADEYI09} 
for general integral operators with oscillatory kernel functions. Indeed, the present analysis of the butterfly algorithm of 
\cite[Sec.~4]{CADEYI09} improves over \cite{demanet-ferrara-maxwell-poulson-ying10} and \cite{kunis-melzer12} 
in that the necessary relation between the underlying polynomial degree $m$ and the depth $L$ of the cluster tree 
is improved from $m \ge C L \log (\log L)$ to $m \ge C \log L$ (cf.~Theorem~\ref{thm:main} and 
Theorem~\ref{thm:helmholtz} in comparision to \cite[Sec.~{4.3}]{demanet-ferrara-maxwell-poulson-ying10}); 
we should, however, put this improvement in perspective by mentioning that a requirement $m \ge C L$ 
typically results from accuracy considerations. 

It is worth stressing that, although techniques that base the transfer between levels 
on polynomial interpolation are amenable to a rigorous stability analysis and may even lead to the desirable 
log-linear complexity, other techniques are available that are observed to perform better in practice. We refer the reader
to \cite{LIYAYI15} for an up-to-date discussion of techniques associated with the name ``butterfly algorithms'' 
and to \cite{BEKUVE15} for techniques related to \emph{adaptive cross approximation} (ACA). 

\medskip
In the present work, we consider kernel functions $k$ of the form 
\begin{equation}
\label{eq:kernel} 
k(x,y) = \exp(\bi \kappa \phase(x,y)) \amp(x,y)
\end{equation}
on a product $B_X \times  B_Y$, where $B_X$, $B_Y\subset\BbbR^d$ are axis-parallel boxes. 
The \emph{phase} function $\phase$ is assumed
to be real analytic on $B_X \times B_Y$, in particular, it is assumed to be
real-valued in this box. The \emph{amplitude}
function  $\amp$ is likewise assumed to be analytic, although it could be complex-valued or even vector-valued. 
The setting of interest for the parameter $\kappa \in \BbbR$  is that of
$\kappa\gg 1$. 

%
An outline of the paper is as follows. We continue the introduction in Section~\ref{sec:notation}
with notation that will be used throughout the paper. In Section~\ref{sec:heart-of-the-matter}, we 
discuss how a butterfly representation for kernels of the form (\ref{eq:kernel}) can be obtained with the aid
of an iterated Chebyshev interpolation procedure. The point of view taken in Section~\ref{sec:heart-of-the-matter}
is one of functions and their approximation in separable form. The following 
Section~\ref{sec:butterfly-on-matrix-level}, therefore, focuses on the realization of the 
butterfly structure on the matrix level. As alluded to in the introduction, butterfly structures are one 
of several techniques for highly oscillatory kernel. In Section~\ref{sec:DH2}, we discuss the  relation
of butterfly techniques with directional ${\mathcal H}^2$-matrices, \cite{BEKUVE15,BO15,boerm-melenk15}. 
The discussion in Section~\ref{sec:heart-of-the-matter} concentrated on the case of analytic 
kernel functions. Often, however, kernel functions are only asymptotically smooth, e.g., if they are 
functions of the Euclidean distance $\|x - y\|$. We propose in 
Section~\ref{sec:butterfly-asymptotically-smooth-kernels} to address this issue by combining the butterfly
structure with a block clustering based on the ``standard'' admissibility condition that takes the distance
of two clusters into account. We mention in passing that alternative approaches are possible to deal
with certain types of singularities (see, e.g., the transform technique advocated in 
\cite[Sec.~{1.1}]{CADEYI09} to account for the Euclidean distance 
$\|x - y\|$). 

Section~\ref{sec:analysis} is concerned with an analysis of the errors incurred by the approximations done
to obtain the butterfly structure, which is enforced by an iterated interpolation process. The stability
of one step of this process is analyzed in the univariate case in Lemma~\ref{lemma:univariate-case}; 
the multivariate case is inferred from that by tensor product arguments in 
Lemma~\ref{lemma:re-interpolation-one-step}. The final stability analysis of the iterated process
is given in Theorem~\ref{thm:nested-interpolation}. 

Section~\ref{sec:helmholtz} specializes the analysis of Theorem~\ref{thm:nested-interpolation} to the 
single layer operator for the Helmholtz equation. This operator is also used in the numerical examples 
in Section~\ref{sec:numerics}.

\subsection{Notation and preliminaries}
\label{sec:notation}

We start with some notation: $B_\varepsilon(z)$ denotes the (Euclidean) ball
of radius $\varepsilon>0$ centered at $z \in \BbbC^d$. For a bounded set $F \subset \BbbR^d$, we denote by 
$\diam_i F$, $i \in \{1,\ldots,d\}$,  
the extent of $F$ in the $i$-th direction: 
\begin{equation}
\label{eq:diam_i}
\diam_i F := \sup \{x_i \,|\, x \in F\} - \inf\{x_i\,|\, x \in F\}.
\end{equation}
For $\rho > 1$ the \emph{Bernstein elliptic discs} are given by 
${\ellipse}_\rho := \{z \in \BbbC\,|\, |z - 1| + |z + 1| < \rho + 1/\rho\}.$   
More generally, for $[a,b] \subset \BbbR$, we also use
the scaled and shifted elliptic discs 
\begin{align*}
{\ellipse}^{[a,b]}_\rho:= \frac{a+b}{2} + \frac{b-a}{2} {\ellipse}_\rho. 
\end{align*}
In a multivariate setting, we collect $d$ values $\rho_i > 1$, $i=1,\dots,d$, in the 
vector $\bs{\rho}$ and define, for intervals $[a_i,b_i] \subset \BbbR$, $i=1,\ldots,d$, the elliptic polydiscs
(henceforth simply called ``polydiscs'') 
\begin{align*}
{\ellipse}_{\bs{\rho}}& := \prod_{i=1}^d {\ellipse}_{\bs{\rho}_i}, 
&  
{\ellipse}^{[\bs{a},\bs{b}]}_{\bs{\rho}}& := \prod_{i=1}^d {\ellipse}^{[a_i,b_i]}_{\bs{\rho}_i}.  
\end{align*}
We will write 
${\ellipse}_{{\rho}}$ and 
${\ellipse}^{[\bs{a},\bs{b}]}_{{\rho}}$ if $\bs{\rho}_i =\rho$ for $i=1,\ldots,d$. 
Axis-parallel boxes are denoted
\begin{equation}
\label{eq:box}
[\bs{a},\bs{b}]:= B^{[\bs{a},\bs{b}]}:= \prod_{i=1}^d [a_i,b_i]. 
\end{equation}
In fact, throughout the text, a \emph{box} $B$ is always understood to be 
a set of the form (\ref{eq:box}). 
For vector-valued objects such as $\bs{\rho}$ we will use the notation $\bs{\rho} > 1$ in a componentwise 
sense. 

We employ a univariate polynomial interpolation operator $I_m:C([-1,1]) \rightarrow {\mathcal P}_m$ with 
Lebesgue constant $\Lambda_m$ (e.g., the Chebyshev interpolation operator). Here, 
${\mathcal P}_m = \operatorname*{span} \{x^i\,|\, 0\leq i\leq m\}$ is the space of (univariate) polynomials of degree $m$. 
We write ${\mathcal Q}_m:= \operatorname*{span} \{x_1^{i_1} x_2^{i_2} \cdots x_d^{i_d}\,|\, 0 \leq i_1,\ldots,i_d \leq m\}$ 
for the space of $d$-variate polynomials of degree $m$ (in each variable). Throughout the text, we will use 
\begin{equation}
\label{eq:M} 
M:= (m+1)^d = \operatorname*{dim} {\mathcal Q}_m. 
\end{equation}
The interpolation operator $I_m$ may be scaled and shifted to an interval
$[a,b] \subset \BbbR$ and is then denoted by $I^{[a,b]}_m$. 
Tensor product interpolation on the 
box $[\bs{a},\bs{b}]$ is correspondingly denoted $I^{[\bs{a},\bs{b}]}_m: C([\bs{a},\bs{b}]) \rightarrow {\mathcal Q}_m$. 
We recall the following error estimates:
\begin{lemma}
\label{lemma:polynomial-interpolation-error}
\begin{align}
\label{eq:error-estimate-tensor-product-interpolation-a} 
\|u - I^{[a,b]}_m[u]\|_{L^\infty([a,b])} &\leq  
( 1 +  \Lambda_m) \inf_{v \in {\mathcal P}_m} \|u - v\|_{L^\infty([a,b])},\\
\label{eq:error-estimate-tensor-product-interpolation} 
\|u - I^{[\bs{a},\bs{b}] }_m[u]\|_{L^\infty([\bs{a},\bs{b}])} &\leq  
\sum_{i=1}^d \Lambda_m^{i-1} (1 + \Lambda_m) \sup_{x \in [\bs{a},\bs{b}]} 
\inf_{v \in {\mathcal P}_m }
\|u_{x,i} - v\|_{L^\infty([a_i,b_i])}, 
\end{align}
where, for given $x \in [\bs{a},\bs{b}]$, 
 the univariate functions $u_{x,i}$ are defined by 
\begin{equation*}
t \mapsto u_{x,i}(t) = u(x_1,\ldots,x_{i-1},t,x_{i+1},\ldots,x_d).
\end{equation*}
\end{lemma}

\noindent
It will sometimes be convenient to write the interpolation operator $I^{[\bs{a},\bs{b}]}_m$ explicitly as 
\begin{equation}
I^{[\bs{a},\bs{b}]}_m[f] = \sum_{i=1}^M f(\xi^{[\bs{a},\bs{b}]}_i) L_{i,[\bs{a},\bs{b}]}, 
\end{equation}
where 
$\xi^{[\bs{a},\bs{b}]}_i$, $i=1,\ldots,M$,
are the interpolation points and
$L_{i,[\bs{a},\bs{b}]}$, $i=1,\ldots,M$,
are the associated Lagrange interpolation polynomials. 
The following lemma is a variant of a result proved in \cite{CADEYI09}:

%
%
\begin{lemma} 
\label{lemma:G}
Let $\Omega \subset \BbbC^{2d}$ be open. Let $(x_0,y_0) \in \Omega$. 
Let $(x,y) \mapsto \widehat \phase(x,y)$ be analytic on $\Omega$.  
Then the function 
\begin{align*}
R_{x_0,y_0} (x,y)
:= \widehat\phase(x,y) - \widehat\phase(x,y_0) - \widehat\phase(x_0,y) + \widehat\phase(x_0,y_0) 
\end{align*}
can be written in the form 
\begin{align*}
R_{x_0,y_0}(x,y)  = (x - x_0) ^\top \widehat G(x,y) (y-y_0), 
\end{align*}
where the entries $\widehat G_{ij}$, $i$, $j=1,\dots,d$, of the matrix $\widehat G$ are analytic on 
$\Omega$. Furthermore, for any convex $K \subset \Omega$ with $(x_0,y_0) \in K$ and 
$d_K:= \sup \{\varepsilon>0\,|\, B_\varepsilon(x) \times B_\varepsilon(y) \subset \Omega
\quad \text{ for all } (x,y) \in K \}>0$ there holds 
\begin{align*}
|\widehat G_{ij}(x,y)| \leq \frac{1}{d_K^2}
\|\widehat\phase\|_{L^\infty(\Omega)} \qquad \text{ for all } (x,y) \in K.  
\end{align*}
\end{lemma}
\begin{proof}
Let us first show the estimate for the convex $K$. 
Let $(x,y)\in K$.
We denote the parametrizations of straight lines from $x_0$ to $x$
and from $y_0$ to $y$ by $x_s := x_0 + s (x-x_0)$ and
$y_t := y_0 + t (y-y_0)$ for $s,t\in[0,1]$.
By integrating along the second line for fixed $x$ and fixed $x_0$ and
along the first line for fixed $y_t$, we arrive at
\begin{align*}
  R_{x_0,y_0}(x,y)
  &= \int_{t=0}^1 \sum_{i=1}^d \left[ \partial_{y_i} \widehat\phase(x, y_t) 
                    - \partial_{y_i}  \widehat\phase(x_0,y_t)\right]
            (y-y_0)_i \,dt\\
  &= \int_{s=0}^1 \int_{t=0}^1 \sum_{j=1}^d \sum_{i=1}^d
         \partial_{x_j} \partial_{y_i} 
         \widehat\phase(x_s, y_t) (x-x_0)_j (y-y_0)_i \,dt \,ds; 
\end{align*}
here, we abbreviated, for example, $(y -y_0)_i$ for the $i$-th component of the vector $(y - y_0)$. 
Using \cite[Cor.~{2.2.5}]{hoermander90} one can show that the double integral indeed represents an
analytic function $K \ni (x,y) \mapsto \widehat{G}(x,y)$. In order to bound $\widehat G$, one observes that the 
integrand involves the partial derivatives of $\widehat\phase$ with respect to the 
variables $x_j$, $y_i$, $i$, $j=1,\ldots,d$, which can be estimated in view of the Cauchy integral formula 
(cf., e.g.,  \cite[Thm.~{2.2.1}]{hoermander90}). 

To see that the coefficients $G_{ij}$ are analytic on $\Omega$ (and not just on $K$), we note that 
we have to define $G_{ij}$ on $\Omega\setminus \operatorname*{closure} K$ as $G_{ij}(x,y) = R_{x_0,y_0}(x,y)/( (x - x_0)_i (y - y_0)_i)$. 
Now that $G_{ij}$ is defined on $\Omega$, we observe that by the analyticity of $R_{x_0,y_0}$, is it analytic
in each variable separately and thus, by Hartogs' theorem (see, e.g., \cite[Thm.~{2.2.8}]{hoermander90}) 
analytic on $\Omega$. 
\end{proof}

\subsection{Butterfly representation: the heart of the matter}
\label{sec:heart-of-the-matter}

Typically, the key ingredient of fast summation schemes is the
approximation of the kernel function by short sums of products of functions
depending only on $x$ or $y$.
Following \cite[Sec.~{4}]{CADEYI09}  we achieve this by applying 
a suitable modification to the kernel function $k$ and then interpolating the remaining smooth
term on a domain $B^X_0\times B^Y_0$, where $B^X_0$ and $B^Y_0$ are
two boxes and $x_0 \in B^X_0$, $y_0 \in B^Y_0$.

According to Lemma~\ref{lemma:G}, we can expect
\begin{equation*}
  R_{x_0,y_0}(x,y) = \phase(x,y) - \phase(x,y_0)
                   - \phase(x_0,y) + \phase(x_0,y_0)
\end{equation*}
to be ``small'' if the product $\|x-x_0\|\,\|y-y_0\|$ is small.
For fixed $x_0$ and $y_0$ we obtain
\begin{align*}
  \exp(\bi \kappa \phase(x,y))
  &= \exp(\bi \kappa \phase(x,y_0))
     \exp(\bi \kappa \phase(x_0,y))
     \exp(-\bi \kappa \phase(x_0,y_0))
     \exp(\bi \kappa R_{x_0,y_0}(x,y))\\
  &= \exp(\bi \kappa \phase(x,y_0))
     \exp(\bi \kappa \phase(x_0,y))
     \exp(\bi \kappa (R_{x_0,y_0}(x,y) - \phase(x_0,y_0))),
\end{align*}
and we observe that the first term on the right-hand side depends only on $x$,
the second one only on $y$, while the third one is smooth if the product
$\kappa \|x-x_0\|\,\|y-y_0\|$ is small, since the modified phase
function $R_{x_0,y_0}(x,y)$ takes only small values and $\phase(x_0,y_0)$
is constant.
This observation allows us to split the kernel function $k$ into
oscillatory factors depending only on $x$ and $y$, respectively, and
a smooth factor $k_{x_0,y_0}$ that can be interpolated:
\begin{equation*}
  k(x,y) = \exp(\bi \kappa \phase(x,y_0))
           \exp(\bi \kappa \phase(x_0,y))
           \underbrace{k(x,y)
                       \exp(-\bi \kappa (\phase(x,y_0) + \phase(x_0,y)))
                      }_{=:k_{x_0,y_0}(x,y)}.
\end{equation*}
Applying the polynomial interpolation operator to $k_{x_0,y_0}$ yields
\begin{align*}
k(x,y) \approx \exp(\bi \kappa \phase(x,y_0)) \exp(\bi \kappa \phase(x_0,y)) 
I^{B^X_0 \times B^Y_0}_m[k_{x_0,y_0}](x,y). 
\end{align*}
Written with the $M = (m+1)^d$ interpolation points $(\xi_p^{B^X_0})_{p=1}^M\subset B^X_0$ and 
$(\xi_q^{B^Y_0})_{q=1}^M \subset B^Y_0$ and the corresponding Lagrange polynomials
$L_{p,B^X_0}$, $L_{q,B^Y_0}$, we have 
\begin{align}
\nonumber 
k(x,y) & \approx 
\sum_{p,q=1}^M \exp(\bi \kappa \phase(x,y_0)) L_{p,B^X_0}(x)
              \exp(\bi \kappa \phase(x_0,y)) L_{q,B^Y_0}(y)\  
              k_{x_0,y_0}(\xi_p^{B^X_0},\xi_q^{B^Y_0})\\
\label{eq:approximation-of-k}
& = 
\sum_{p,q=1}^M L^x_{p,B^X_0,y_0}(x) L^y_{q,B^Y_0,x_0}(y) \ 
 k_{x_0,y_0}(\xi_p^{B^X_0},\xi_q^{B^Y_0}),
\end{align}
where the \emph{expansion functions} $L^x_{p,B^X_0,y_0}$ and $L^y_{q,B^Y_0,x_0}$ are defined by 
\begin{align}
\label{eq:widetilde-L}
L^x_{p,B^X_0,y_0}  &= 
\exp(\bi \kappa \phase(\punkt,y_0)) L_{p,B^X_0}, 
& L^y_{q,B^Y_0,x_0}  &= 
\exp(\bi \kappa \phase(x_0,\punkt)) L_{q,B^Y_0}. 
\end{align}
A short form of the approximation is given by 
\begin{equation}
\label{eq:middle-level-approximation-analysis}
\left({\mathfrak I}^{B^X_0,x}_{y_0}  \otimes {\mathfrak I}^{B^Y_0,y}_{x_0} \right)[k],
\end{equation}
where, for a box $B\subset \BbbR^d$, a point $z \in \BbbR^d$, and a polynomial degree $m$, 
we have introduced the operators
\begin{align*}
{\mathfrak I}^{B,x}_{z}[f] &:= 
\exp(\bi \kappa \phase(\punkt,z))
   I^{B}_{m} [ \exp(- \bi \kappa \phase(\punkt,z)) f ], 
\qquad \qquad 
{\mathfrak I}^{B,y}_{z}[f] &:= 
\exp(\bi \kappa \phase(z,\punkt))
   I^{B}_{m} [ \exp(- \bi \kappa \phase(z,\punkt)) f ].
\end{align*}
We have seen that the product $\kappa\|x-x_0\|\,\|y-y_0\|$ controls the
smoothness of $k_{x_0,y_0}$, so we can move $y$ away from $y_0$ as
long as we move $x$ closer to $x_0$ without changing the quality
of the approximation.
This observation gives rise to a \emph{multilevel} approximation 
of (\ref{eq:approximation-of-k}) that applies
an additional re-interpolation step to the expansion functions $L^x_{p,B^X_0,y_0}$
and $L^y_{q,B^Y_0,x_0}$.
We only describe the process for $L^x_{p,B^X_0,y_0}$, since 
$L^y_{q,B^Y_0,x_0}$ is treated analogously. 
Let $(B^X_\ell)_{\ell=0}^L$ be a nested sequence of boxes 
and $(y_{-\ell})_{\ell=0}^{L}$ be a sequence of points. The first step of the iterated re-interpolation process 
is given by 
\begin{align}
\label{eq:reinterpolation-one-step}
L^x_{p,B^X_0,y_0}|_{B^X_1} &\approx \exp(\bi \kappa \phase(\punkt,y_{-1}) 
I^{B^X_1}_m \left[ L^x_{p,B^X_0,y_0} \exp(-\bi \kappa \phase(\punkt,y_{-1}))\right]
= {\mathfrak I}^{B^X_1,x}_{y_{-1}}[L^x_{p,B^X_0,y_0}].
\end{align}
The actual approximation of $L^x_{p,B^X_0,y_0}$ is then given by iteratively applying
the operators ${\mathfrak I}^{B^X_\ell,x}_{y_{-\ell}}$, $\ell=0,\ldots,L$, i.e., 
\begin{align}\label{eq:lagrange_reinterpolation}
L^x_{p,B^X_0,y_0}|_{B_L^X}  &\approx   {\mathfrak I}^{B^X_L,x}_{y_{-L}} \circ \cdots 
\circ {\mathfrak I}^{B_1^X,x}_{y_{-1}} 
[L^x_{p,B^X_0,y_0}]
\end{align}
This process is formalized in the following algorithm: 

%
%
\begin{algorithm}[Butterfly representation by interpolation]
\label{alg:butterfly-analysis}
Let two sequences 
$(B_\ell^X)_{\ell=0}^L$, 
$(B_\ell^Y)_{\ell=0}^L$, of nested boxes and sequences of points 
$(y_{-\ell})_{\ell=0}^{L}$, $(x_{-\ell})_{\ell=0}^{L}$ be given.
The \emph{butterfly representation of $k$ on $B^X_L \times B^Y_L$} is
defined by 
\begin{equation*}
k^{BF}:=  
\left( {\mathfrak I}^{B_L^X,x}_{y_{-L}} \circ \cdots
       \circ {\mathfrak I}^{B_0^X,x}_{y_0} \right) \otimes 
\left( {\mathfrak I}^{B_L^Y,y}_{x_{-L}} \circ \cdots
       \circ {\mathfrak I}^{B_0^Y,y}_{x_0} \right) [k].
\end{equation*}
\end{algorithm}
In Theorem~\ref{thm:main} below we will quantify the error 
$k|_{B^X_L \times B^Y_L} - k^{BF}$. 

%
%
\begin{remark}
In Algorithm~\ref{alg:butterfly-analysis}, the number of levels $L$ is chosen to be the 
same for the first argument ``$x$'' and the second argument ``$y$''.
The above developments
show that this is not essential, and Algorithm~\ref{alg:butterfly-analysis} naturally generalizes
to a setting with boxes $(B_\ell^X)_{\ell=0}^{L^X}$, $(B_\ell^Y)_{\ell=0}^{L^Y}$
and corresponding point sequences $(y_{-\ell})_{\ell=0}^{L^X}$, $(x_{-\ell})_{\ell=0}^{L^Y}$ 
with $L^X \ne L^Y$. 
\eremk
\end{remark}

\subsection{Butterfly structures on the matrix level}
\label{sec:butterfly-on-matrix-level} 

It is instructive to formulate how the approximation described in Algorithm~\ref{alg:butterfly-analysis} 
is realized on the matrix level. To that end, we consider the Galerkin
discretization of an integral operator $K: L^2(\Gamma) \rightarrow L^2(\Gamma^\prime)$ defined by 
$(K \varphi)(x) := \int_{y \in \Gamma} k(x,y) \varphi(y)$.  Let 
$(\varphi_i)_{i \in \Idx} \subset L^2(\Gamma^\prime)$, $(\psi_j)_{j \in \Idy} \subset L^2(\Gamma)$ 
be bases of trial and test spaces. We have to represent the Galerkin matrix ${\mathbf K}$ with entries 
\begin{equation}
\label{eq:stiffness-matrix}
{\mathbf K}_{i ,j } = \int_{x \in \Gamma^\prime}\int_{y \in \Gamma} k(x,y)
   \varphi_i(x)\psi_j(y) \,dy\,dx,
\qquad i \in \Idx, \quad j \in \Idy.
\end{equation}
Following the standard approach for fast multipole \cite{RO85} and
panel-clustering methods \cite{HANO89}, the index sets $\Idx$ and $\Idy$
are assumed to be organized in \emph{cluster trees} ${\mathcal T}_\Idx$
and ${\mathcal T}_\Idy$, where the nodes of the tree are called \emph{clusters}.
We assume that the maximal number of sons of a cluster is fixed.
A cluster tree ${\mathcal T}_\Idx$ can be split into levels; the root, which
is the complete index set $\Idx$, is at level $0$. We employ the notation 
${\mathcal T}^\ell_\Idx$ and ${\mathcal T}^\ell_\Idy$ for the clusters on level $\ell$. 
We use the notation $\sons(\sigma)$ for the collection of sons of a cluster $\sigma$ (for leaves $\sigma$, 
we set $\sons(\sigma) = \emptyset$). We let  $\father(\sigma)$ denote the (unique) father of a cluster $\sigma$ 
that is not the root. We use $\descendants(\sigma)$ for the set of descendants of the cluster $\sigma$ 
(including $\sigma$ itself).

A tuple $(\sigma_0,\sigma_1,\ldots,\sigma_n)$ of clusters is called a
\emph{cluster sequence} if
\begin{equation*}
  \sigma_{\ell+1}\in\sons(\sigma_\ell)
  \qquad\text{ for all } \ell\in\{0,\ldots,n-1\}.
\end{equation*}
We introduce for each cluster $\sigma \in {\mathcal T}_\Idx$
a \emph{bounding box} $B_\sigma$, which is an axis-parallel box such that 
\begin{equation}
\label{eq:bounding-box}
\operatorname*{supp} \varphi_i \subset B_\sigma \qquad \text{ for all } i \in \sigma,
\end{equation}
and similarly for clusters $\tau\in{\mathcal T}_\Idy$ and basis
functions $\psi_j$.

\subsubsection{Butterfly structure in a model situation}
\label{sec:butterfly-model}

We illustrate the butterfly structure (based on interpolation as
proposed in \cite[Sec.~4]{CADEYI09}) for a model situation, where 
the leaves of the cluster trees ${\mathcal T}_\Idx$ and ${\mathcal T}_\Idy$
are all on the same level.
In particular, this assumption implies
$\depth({\mathcal T}_\Idx) = \depth({\mathcal T}_\Idy)$.

We fix points $x_\sigma\in B_\sigma$ and $y_\tau\in B_\tau$ for all
$\sigma\in\mathcal{T}_\Idx$ and $\tau\in\mathcal{T}_\Idy$.

For a given pair $(\sigma,\tau) \in\mathcal{T}_\Idx \times \mathcal{T}_\Idy$,
combining the intermediate decomposition (\ref{eq:approximation-of-k})
with (\ref{eq:stiffness-matrix}) yields
\begin{align*}
  \mathbf{K}_{i,j}
  &\approx \int_{\Gamma'} \int_\Gamma \sum_{p=1}^M \sum_{q=1}^M
          L^x_{p,B_\sigma,y_\tau}(x)
          k_{x_\sigma,y_\tau}(\xi_p^{B_\sigma},\xi_q^{B_\tau})
          L^y_{q,B_\tau,x_\sigma}(y)
          \varphi_i(x) \psi_j(y) \,dy\,dx\\
  &= \sum_{p=1}^M \sum_{q=1}^M
     \underbrace{
       \int_{\Gamma'} L^x_{p,B_\sigma,y_\tau}(x) \varphi_i(x) \,dx
     }_{=:{\mathbf V}^{\sigma,\tau}_{i,p}}
     \underbrace{
       k_{x_\sigma,y_\tau}(\xi_p^{B_\sigma},\xi_q^{B_\tau})
     }_{=:{\mathbf S}^{\sigma\times\tau}_{p,q}}
     \underbrace{
       \int_\Gamma L^y_{q,B_\tau,x_\sigma}(y) \psi_j(y) \,dy
     }_{=:{\mathbf W}^{\tau,\sigma}_{j,q}}\\
  &= (\mathbf{V}^{\sigma,\tau} \mathbf{S}^{\sigma\times\tau}
      (\mathbf{W}^{\tau,\sigma})^\top)_{i,j}
   \qquad\text{ for all } i\in\sigma,\ j\in\tau.
\end{align*}
If $\sigma$ is not a leaf of $\mathcal{T}_\Idx$, we employ the additional
approximation (\ref{eq:lagrange_reinterpolation}).
We choose a \emph{middle level} $L\in\BbbN_0$ and cluster sequences
$(\sigma_{-L},\ldots,\sigma_0,\ldots,\sigma_L)$ and
$(\tau_{-L},\ldots,\tau_0,\ldots,\tau_L)$ of clusters
such that $\sigma_0=\sigma$, $\tau_0=\tau$.
Since each cluster has at most one father, the clusters
$\sigma_{-L},\ldots,\sigma_0$ are uniquely determined by $\sigma=\sigma_0$
and $\tau_{-L},\ldots,\tau_0$ are uniquely determined by $\tau=\tau_0$.
The approximation (\ref{eq:lagrange_reinterpolation}) implies for $i \in \sigma_L$ 
\begin{equation*}
  {\mathbf V}^{\sigma,\tau}_{i,p}
  \approx \int_{\Gamma'} \mathfrak{I}^{B_{\sigma_L},x}_{y_{\tau_{-L}}}
      \circ \cdots \circ
      \mathfrak{I}^{B_{\sigma_1},x}_{y_{\tau_{-1}}}
      [L^x_{p,B_{\sigma_0},y_{\tau_0}}](x)
      \varphi_i(x) \,dx, 
   \qquad i\in\sigma_L,\ p\in\{1,\ldots,M\}.
\end{equation*}
The first interpolation step is given by
\begin{align}
\nonumber 
   \mathfrak{I}^{B_{\sigma_1},x}_{y_{\tau_{-1}}}
  [L_{p,B_{\sigma_0},y_{\tau_0}}]  
  &= \exp(\bi \kappa \Phi(x,y_{\tau_{-1}}))
       \sum_{n=1}^M \exp(\bi \kappa (\Phi(\xi^{B_{\sigma_1}}_n,y_{\tau_0})
                        - \Phi(\xi^{B_{\sigma_1}}_n,y_{\tau_{-1}})))
   L_{p,B_{\sigma_0}}(\xi^{B_{\sigma_1}}_n)
                L_{n,B_{\sigma_1}} \\
\label{eq:re-expansion-foo}
  &= \sum_{n=1}^M {\mathbf E}^{\sigma_1,\sigma_0,\tau_0,\tau_{-1}}_{n,i}
                L_{n,B_{\sigma_1},y_{\tau_{-1}}}
    \qquad\text{ for all } p\in\{1,\ldots,M\},
\end{align}
where the \emph{transfer matrix}
$\mathbf{E}^{\sigma_1,\sigma_0,\tau_0,\tau_{-1}}$
is given by
\begin{equation*}
  {\mathbf E}^{\sigma_1,\sigma_0,\tau_{-1},\tau_0}_{n,i}
  = \exp(\bi \kappa (\Phi(\xi^{B_{\sigma_1}}_n,y_{\tau_0})
                     - \Phi(\xi^{B_{\sigma_1}}_n,y_{\tau_{-1}})))
    L_{p,B_{\sigma_0}}(\xi^{B_{\sigma_1}}_n)
\qquad \mbox{for all $n$, $i\in\{1,\ldots,M\}$.}
\end{equation*}
The re-expansion (\ref{eq:re-expansion-foo}) implies 
\begin{equation*}
  {\mathbf V}^{\sigma_\ell,\tau_{-\ell}}_{i,p}
  = \sum_{n=1}^M {\mathbf V}^{\sigma_{\ell+1},\tau_{-\ell-1}}_{i,n}
    {\mathbf E}^{\sigma_{\ell+1},\sigma_\ell,\tau_{-\ell-1},\tau_{-\ell}}_{n,p}
  \qquad\text{ for all } i\in\sigma_{\ell+1},\ p\in\{1,\ldots,M\},
\end{equation*}
so we can avoid storing $\mathbf{V}^{\sigma_\ell,\tau_{-\ell}}$ by storing
the $M\times M$-matrix
$\mathbf{E}^{\sigma_{\ell+1},\sigma_\ell,\tau_{-\ell},\tau_{-\ell-1}}$
and $\mathbf{V}^{\sigma_{\ell+1},\tau_{-\ell-1}}$ instead.
A straightforward induction yields that we only have to store
$\mathbf{V}^{\sigma_L,\tau_{-L}}$ and the transfer matrices.

%
%
\begin{algorithm}[Butterfly representation of matrices]
Let $L = \lfloor \depth(\mathcal{T}_\Idx)/2 \rfloor$ and
$L^\text{middle} = \depth(\mathcal{T}_\Idx)-L
  = \lceil \depth(\mathcal{T}_\Idx)/2 \rceil \geq L$.
\begin{enumerate} 
\item Compute,  for all $\sigma\in\mathcal{T}_\Idx^{L^\text{middle}}$, $\tau \in\mathcal{T}_\Idy^{L^\text{middle}}$,
the \emph{coupling matrices}
  \begin{equation*}
    {\mathbf S}^{\sigma \times \tau}_{p,q}
    = k_{x_\sigma,y_\tau}(\xi_p^{B_\sigma},\xi_q^{B_\tau}), \qquad 
  p,q\in\{1,\ldots,M\}.
  \end{equation*}
\item Compute the \emph{transfer matrices}
  \begin{equation*}
    {\mathbf E}^{\sigma_{\ell+1},\sigma_\ell,\tau_{-\ell-1},\tau_{-\ell}}_{n,p}
    = \exp(\bi \kappa (\Phi(\xi_n^{B_{\sigma_{\ell+1}}},y_{\tau_{-\ell}})
                - \Phi(\xi_n^{B_{\sigma_{\ell+1}}},y_{\tau_{-\ell-1}})))
      L_{p,B_{\sigma_\ell}}(\xi_n^{B_{\sigma_{\ell+1}}}), 
\qquad \mbox{$n,p\in\{1,\ldots,M\}$}, 
  \end{equation*}
  for all $\ell\in\{0,\ldots,L-1\}$,
  $\sigma_\ell\in\mathcal{T}_\Idx^{L^\text{middle}+\ell}$,
  $\sigma_{\ell+1}\in\sons(\sigma_\ell)$,
  $\tau_{-\ell-1}\in\mathcal{T}_\Idy^{L^\text{middle}-\ell-1}$,
  $\tau_{-\ell}\in\sons(\tau_{-\ell-1})$,
and
  \begin{equation*}
    {\mathbf E}^{\tau_{\ell+1},\tau_\ell,\sigma_{-\ell-1},\sigma_{-\ell}}_{n,q}
    = \exp(\bi \kappa (\Phi(x_{\sigma_{-\ell}},\xi_n^{B_{\tau_{\ell+1}}})
                - \Phi(x_{\sigma_{-\ell-1}},\xi_n^{B_{\tau_{\ell+1}}})))
      L_{q,B_{\tau_\ell}}(\xi_n^{B_{\tau_{\ell+1}}}), 
  \qquad \mbox{$n,q\in\{1,\ldots,M\}$}
  \end{equation*}
  for all $\ell\in\{0,\ldots,L-1\}$,
  $\sigma_{-\ell-1}\in\mathcal{T}_\Idx^{L^\text{middle}-\ell-1}$,
  $\sigma_{-\ell}\in\sons(\sigma_{-\ell-1})$,
  $\tau_\ell\in\mathcal{T}_\Idy^{L^\text{middle}+\ell}$,
  $\tau_{\ell+1}\in\sons(\tau_\ell)$.
\item Compute the \emph{leaf matrices}
  \begin{equation*}
    {\mathbf V}^{\sigma_L,\tau_{-L}}_{i,p}
    = \int_{\Gamma'}
      \exp(\bi \kappa \Phi(x, y_{\tau_{-L}}))
      L_{p,B_{\sigma_L}}(x)
      \varphi_i(x) \,dx
  \end{equation*}
  for all $\sigma_L\in\mathcal{T}^{L^\text{middle}+L}_\Idx$,
  $\tau_{-L}\in\mathcal{T}^{L^\text{middle}-L}_\Idy$,
  $i\in\sigma_L$, $p\in\{1,\ldots,M\}$, and
  \begin{equation*}
    {\mathbf W}^{\tau_L,\sigma_{-L}}_{j,q}
    = \int_\Gamma
      \exp(\bi \kappa \Phi(x_{\sigma_{-L}}, y))
      L_{q,B_{\tau_L}}(y)
      \psi_j(y) \,dy
  \end{equation*}
  for all $\tau_L\in\mathcal{T}^{L^\text{middle}+L}_\Idy$,
  $\sigma_{-L}\in\mathcal{T}^{L^\text{middle}-L}_\Idx$,
  $j\in\tau_L$, $q\in\{1,\ldots,M\}$.
\item For leaf clusters $\sigma_L\in\mathcal{T}_\Idx$
  and $\tau_L\in\mathcal{T}_\Idy$ there are uniquely determined
  cluster sequences $(\tau_{-L},\ldots,\tau_0,\ldots,\tau_L)$
  and $(\sigma_{-L},\ldots,\sigma_0,\ldots,\sigma_L)$.
  The matrix $\mathbf{K}$ is approximated by
  \begin{align*}
    & \mathbf{K}|_{\sigma_L\times\tau_L} \approx \\
    & \mathbf{V}^{\sigma_L,\tau_{-L}}
             \mathbf{E}^{\sigma_L,\sigma_{L-1},\tau_{-L},\tau_{-L+1}}
             \cdots
             \mathbf{E}^{\sigma_1,\sigma_0,\tau_{-1},\tau_0}
             \mathbf{S}^{\sigma_0\times\tau_0}
(\mathbf{E}^{\tau_1,\tau_0,\sigma_{-1},\sigma_0})^\top
             \cdots
            (\mathbf{E}^{\tau_L,\tau_{L-1},\sigma_{-L},\sigma_{-L+1}})^\top
            (\mathbf{W}^{\tau_L,\sigma_{-L}})^\top.
  \end{align*}
\end{enumerate}
\end{algorithm}

%
%
\begin{remark}
The costs of representing a butterfly matrix are as follows (for even $\operatorname*{depth} {\mathcal T}_\Idx$
and $L^\text{middle} = L = \operatorname*{depth}({\mathcal T}_\Idx/2$):
\begin{enumerate}
\item 
For the coupling matrices ${\mathbf S}^{\sigma\times\tau}$ on the middle level $L^\text{middle}$:
$M^2 |{\mathcal T}^{L^\text{middle}}_\Idx| |{\mathcal T}^{L^\text{middle}}_\Idy|$ 
\item For the transfer matrices $\mathbf{E}^{\sigma_{\ell+1},\sigma_\ell,\tau_{-\ell-1},\tau_\ell}$:
$\sum_{\ell=0}^{L-1} M^2 |{\mathcal T}^{L^\text{middle}+\ell}_\Idx| |{\mathcal T}^{L^\text{middle}-\ell}_\Idy|$ 
\item For the transfer matrices $\mathbf{E}^{\tau_{\ell+1},\tau_\ell,\sigma_{-\ell-1},\sigma_{-\ell}}$:
$\sum_{\ell=0}^{L-1} M^2 |{\mathcal T}^{L^\text{middle}-\ell}_\Idx| |{\mathcal T}^{L^\text{middle}+\ell}_\Idy|$ 
\item For the leaf matrices ${\mathbf V}^{\sigma_L,\tau_{-L}}$ and 
${\mathbf W}^{\tau_L,\sigma_{-L}}$ with leaves $\sigma_L\in\mathcal{T}_\Idx$
and $\tau_L\in\mathcal{T}_\Idy$: 
$\sum_{\sigma \in {\mathcal T}_\Idx} M |\sigma|$ and 
$\sum_{\tau\in {\mathcal T}_\Idy} M |\tau|$.
\end{enumerate}
In a model situation with $|\Idx| = |\Idy| = N$ and balanced binary trees ${\mathcal T}_\Idx$, ${\mathcal T}_\Idy$
of depth $2L = O(\log N)$ and leaves of ${\mathcal T}_\Idx$, ${\mathcal T}_\Idy$ that have at most $n_\mathrm{leaf}$ elements
we get 
\begin{equation*}
  M^2 \sqrt{N} \sqrt{N} + 2 M^2 L N + 2 n_\mathrm{leaf} N
  = O(M^2 \sqrt{N} \sqrt{N} + 2 M^2 N \log N  + 2 n_\mathrm{leaf} N). 
\end{equation*}
We expect for approximation-theoretical reasons that the polynomial 
degree satisfies $m = O(\log N)$. Since $M = (m+1)^d$, the total complexity is then $O(N \log^{2d+1} N)$. 
\eremk 
\end{remark}

%
%
\begin{remark}
The butterfly structure presented here is suitable for kernel functions with analytic phase function 
$\phase$ and amplitude function $\amp$. When these functions are ``asymptotically smooth'', for example, 
when they are functions of the Euclidean distance $(x,y) \mapsto \|x - y\|$, a
modification is necessary to take care of the singularity at $x=y$. 
For example, one could create a block partition that applies the
approximation scheme only to pairs $(\sigma,\tau)$ of clusters that satisfy the
\emph{standard admissibility condition}
$\max\{\diam(B_\sigma),\diam(B_\tau)\}\leq\dist(B_\sigma,B_\tau)$.
Each block $\mathbf{K}|_{\sigma\times\tau}$ that satisfies this condition is
treated as a butterfly matrix in the above sense.
We illustrate this procedure in 
Section~\ref{sec:butterfly-asymptotically-smooth-kernels} below for general asymptotically smooth kernel
functions $k$ and specialize to the 3D Helmholtz kernel in Section~\ref{sec:helmholtz}.
\eremk
\end{remark}

\subsubsection{${\mathcal D}{\mathcal H}^2$-matrices}
\label{sec:Dh2-matrices}

It is worth noting that the above butterfly structure can be interpreted as a special case 
of directional ${\mathcal H}^2$-matrices (short: ${\mathcal D}{\mathcal H}^2$-matrices) as introduced 
in \cite{BEKUVE15,BO15,boerm-melenk15} in the context of discretizations
of Helmholtz integral operators. 

Let us recall the definition of a
${\mathcal D}{\mathcal H}^2$-matrix ${\mathbf K} \in \BbbC^{\Idx\times\Idy}$
with cluster trees ${\mathcal T}_\Idx$, ${\mathcal T}_\Idy$.  

%
%
\begin{definition}[Directional cluster basis for ${\mathcal T}_\Idx$]
\label{definition:cluster-basis}
For each cluster $\sigma \in {\mathcal T}_\Idx$, let ${\mathcal D}_\sigma$
be a given index set. 
Let ${\mathcal V} = ({\mathbf V}^{\sigma,c})_{\sigma \in {\mathcal T}_\Idx, c \in {\mathcal D}_\sigma}$ 
be a two-parameter family of matrices. This family is called a \emph{directional cluster basis} with \emph{rank $M$} 
if
\begin{itemize}
  \item ${\mathbf V}^{\sigma, c}\in\BbbC^{\sigma \times M}$ for all $\sigma \in{\mathcal T}_\Idx$
     and $c\in\mathcal{D}_\sigma$, and
  \item there is, for every $\sigma$ that is not a leaf of ${\mathcal T}_\Idx$ and every
  $\sigma^\prime \in \sons(\sigma)$ and every $c \in {\mathcal D}_\sigma$, an element $c' \in {\mathcal D}_{\sigma'}$  
  and a matrix ${\mathbf E}^{\sigma',\sigma,c',c} \in \BbbC^{M \times M}$ such that 
     \begin{align}\label{eq:nested}
       {\mathbf V}^{\sigma,c}|_{\sigma'\times \{1,\ldots,M\}} &= {\mathbf V}^{\sigma', c'} {\mathbf E}^{\sigma',\sigma,c',c}. 
     \end{align}
\end{itemize}
The matrices ${\mathbf E}^{\sigma',\sigma,c',c}$ are called \emph{transfer matrices} for
the directional cluster basis. 
\end{definition}

${\mathcal D}{\mathcal H}^2$-matrices are blockwise low-rank matrices. To describe the details of this
structure, let ${\mathcal T}_{\Idx \times \Idy}$ be a block tree based on the cluster trees 
${\mathcal T}_\Idx$  and ${\mathcal T}_\Idy$. 
Specifically, we assume that a) the root of ${\mathcal T}_{\Idx \times \Idy}$ is $\Idx \times \Idy$, b) each
node of ${\mathcal T}_{\Idx \times \Idy}$ is of the form 
$(\sigma,\tau) \in {\mathcal T}_\Idx \times {\mathcal T}_\Idy$, and c) for every node 
$(\sigma,\tau) \in {\mathcal T}_{\Idx \times \Idy}$ we have 
\begin{equation*}
\sons((\sigma,\tau)) \ne \emptyset  \quad \Longrightarrow \quad 
\sons((\sigma,\tau))  = \sons(\sigma) \times \sons(\tau). 
\end{equation*}
We denote the \emph{leaves} of the block tree ${\mathcal T}_{\Idx\times\Idy}$ by
\begin{equation*}
  \lfIJ := \{ b\in{\mathcal T}_{\Idx\times\Idy}\ :\ \sons(b)=\emptyset \}.
\end{equation*}
The leaves form a disjoint partition of $\Idx\times\Idy$, so a matrix ${\mathbf G}$ is uniquely determined
by the submatrices ${\mathbf G}|_{\sigma\times\tau}$ for $b=(\sigma,\tau)\in\lfIJ$. The set of leaves 
$\lfIJ$ is written as the disjoint union 
$\lfaIJ\dot\cup\lfiIJ$ of two sets, which are called the 
\emph{admissible} leaves 
$\lfaIJ$, corresponding to submatrices that can be approximated,
and the \emph{inadmissible} leaves $\lfiIJ$, corresponding to small
submatrices that have to be stored explicitly.
We are now in a position to define $\mathcal{DH}^2$-matrices as in 
\cite{BO15,boerm-melenk15}: 

%
%
\begin{definition}[Directional $\mathcal{H}^2$-matrix]
Let ${\mathcal V}$ and ${\mathcal W}$ be directional cluster bases of rank $M$ for ${\mathcal T}_\Idx$ 
and ${\mathcal T}_\Idy$, respectively. 
A matrix ${\mathbf G}\in\BbbC^{\Idx\times\Idy}$ 
is called a \emph{directional $\mathcal{H}^2$-matrix} (or simply: a
\emph{$\mathcal{DH}^2$-matrix}) if there are families
${\mathcal S}=({\mathbf S}_b)_{b\in\lfaIJ}$ and $(c_b^\sigma)_{b\in\lfaIJ}$, $(c_b^\tau)_{b \in \lfaIJ}$ 
such that
\begin{itemize}
  \item ${\mathbf S}_b\in\BbbC^{M\times M}$ 
    for all $b=(\sigma,\tau)\in\lfaIJ$, and
  \item ${\mathbf G}|_{\sigma\times\tau} = 
          {\mathbf V}^{\sigma, c^\sigma_b} {\mathbf S}_b ({\mathbf W}^{\tau,c^\tau_b})^\top$ 
    with $c^\sigma_b\in\mathcal{D}_\sigma$, $c^\tau_b\in\mathcal{D}_\tau$
    for all $b=(\sigma,\tau)\in\lfaIJ$.
\end{itemize}
The elements of the family ${\mathcal S}$ are called \emph{coupling matrices}.
The cluster bases ${\mathcal V}$ and ${\mathcal W}$ are called the \emph{row cluster basis}
and \emph{column cluster basis}, respectively.
A \emph{$\mathcal{DH}^2$-matrix representation} of a
$\mathcal{DH}^2$-matrix ${\mathbf G}$ consists of ${\mathcal V}$, ${\mathcal W}$, ${\mathcal S}$ and the
family $({\mathbf G}|_{\sigma\times\tau})_{b=(\sigma,\tau)\in\lfiII}$ of \emph{nearfield matrices} corresponding
to the inadmissible leaves of ${\mathcal T}_{\Idx\times\Idy}$.
\end{definition}

\subsubsection{The butterfly structure as a special $\mathcal{DH}^2$-matrix}
\label{sec:DH2}

We now show that the butterfly structure discussed in Section~\ref{sec:butterfly-model} 
can be understood as a ${\mathcal D}{\mathcal H}^2$-matrix:
for $L=\lfloor \depth(\mathcal{T}_\Idx)/2 \rfloor$ and the
middle level $L^\text{middle} := \depth(\mathcal{T}_\Idx) - L
= \lceil \depth(\mathcal{T}_\Idx)/2 \rceil$, we let
\begin{equation*}
  \mathcal{L}_{\Idx\times\Idy}^+
  = \{ (\sigma,\tau)\ :\ \sigma\in\mathcal{T}^{L^\text{middle}}_\Idx,
                         \tau\in\mathcal{T}^{L^\text{middle}}_\Idy \},
  \qquad
  \mathcal{L}_{\Idx\times\Idy}^-
  = \emptyset.
\end{equation*}
The key is to observe that the sets ${\mathcal D}_\sigma$ associated with 
a cluster $\sigma \in {\mathcal T}^{L^\text{middle}+\ell}_\Idx$ on level $L^\text{middle}+\ell$
are taken to be points $y_\tau$ with $\tau \in {\mathcal T}^{L^\text{middle}-\ell}_\Idy$: 
\begin{equation}
{\mathcal D}_\sigma := \{y_\tau\,|\, \tau \in {\mathcal T}^{L^\text{middle}-\ell}_\Idy\} \qquad 
\text{ for } \quad \sigma \in {\mathcal T}^{L^\text{middle}+\ell}_\Idx,\ \ell=0,\ldots,L. 
\end{equation}
Analogously, we define
the sets ${\mathcal D}_\tau$ for $\tau \in {\mathcal T}_\Idy$. 
The transfer matrices ${\mathbf E}^{\sigma',\sigma,c',c}$ appearing in
Definition~\ref{definition:cluster-basis} are those of
Section~\ref{sec:butterfly-model}, and the same holds for
the leaf matrices $\mathbf{V}^{\sigma_L,c}$ and $\mathbf{W}^{\tau_L,c}$.

\subsubsection{Butterfly structures for asymptotically smooth kernels in a model situation}
\label{sec:butterfly-asymptotically-smooth-kernels}
The model situation of Section~\ref{sec:butterfly-model} is appropriate when the kernel function $k$ 
is analytic. Often, however, the kernel function $k$ is only ``asymptotically smooth'', i.e., it satisfies 
estimates of the form 
\begin{equation}
\label{eq:asymptotically-smooth} 
|D_x^\alpha D_y^\beta k(x,y)|  \leq C \frac{\alpha! \beta!}{\|x - y\|^{|\alpha| + |\beta| + \delta}} \gamma^{|\alpha|+|\beta|} 
\qquad \text{ for all } \alpha, \beta \in \BbbN_0^d
\end{equation}
for some $C$, $\gamma > 0$, $\delta \in \BbbR$. Prominent examples include kernel function such as the 
3D Helmholtz kernel (\ref{eq:helmholtz-kernel}) where the dependence on $(x,y)$ is through the Euclidean distance $\|x - y\|$.  

We describe the data structure for an approximation of the stiffness matrix ${\mathbf K}$ given by (\ref{eq:stiffness-matrix}) 
for asymptotically smooth kernels. We will study a restricted setting that focuses on the essential points and is geared
towards kernels such as the Helmholtz kernel (\ref{eq:helmholtz-kernel}). 

Again, let the index sets $\Idx$ and $\Idy$  be organized in
trees ${\mathcal T}_\Idx$ and ${\mathcal T}_\Idy$ with a bounded number
  of sons.
We assume that the trees are balanced and that all leaves are on the same level
$\depth({\mathcal T}_\Idx) = \depth({\mathcal T}_\Idy)$. 
Recall the notion of bounding box in (\ref{eq:bounding-box}). 
It will also be convenient to introduce for $\sigma \in {\mathcal T}_\Idx$ and $\tau \in {\mathcal T}_\Idy$ the 
\emph{subtrees} 
$$
{\mathcal T}_\Idx(\sigma) \quad  \mbox{ and } \quad 
{\mathcal T}_\Idy(\tau) 
$$
with roots $\sigma$ and $\tau$, respectively, and the clusters on level $\ell$:
$$
{\mathcal T}^\ell_\Idx(\sigma) := 
{\mathcal T}_\Idx(\sigma) \cap {\mathcal T}^\ell_\Idx, 
\qquad 
{\mathcal T}^\ell_\Idy(\tau) := {\mathcal T}_\Idy(\tau)  \cap {\mathcal T}^\ell_\Idy.
$$
Concerning the block cluster tree ${\mathcal T}_{\Idx \times \Idy}$, we assume that its 
leaves $\lfIJ = \lfaIJ \dot \cup \lfiIJ$ are created as follows: 
\begin{enumerate}
\item 
Apply a clustering algorithm to create a block tree ${\mathcal T}^{standard}_{\Idx \times \Idy}$ based on the 
trees ${\mathcal T}_\Idx$ and ${\mathcal T}_\Idy$ 
according to the \emph{standard admissibility condition}
\begin{equation}
\label{eq:eta_1}
\operatorname*{max} (\operatorname*{diam} B_\sigma, \operatorname*{diam} B_\tau) 
\leq \eta_1 
\operatorname*{dist} (B_\sigma, B_\tau) 
\end{equation}
for a fixed admissibility parameter $\eta_1 > 0$. 
\item 
The leaves of ${\mathcal T}^{standard}_{\Idx \times \Idy}$ are split as 
$\lfsaIJ \dot\cup \lfsiIJ$ 
into \emph{admissible} leaves 
$\lfsaIJ$ and \emph{inadmissible} leaves $\lfsiIJ$. 
\item
Set $\lfiIJ:= \lfsiIJ$.
\item 
For all $(\widehat{\sigma},\widehat{\tau})\in\lfsaIJ$, we define
a set $\lfpa(\sigma,\tau)$ of sub-blocks satisfying a stronger
admissibility condition:
Let $(\widehat{\sigma},\widehat{\tau})\in\lfsaIJ$ and
$\ell:=\level(\widehat{\tau})$.
Our assumptions imply $\ell=\level(\widehat{\sigma})$.

Set $L_\ell := \lfloor (\depth(\mathcal{T}_\Idx) - \ell) / 2 \rfloor$
and $L_{\ell}^\text{middle} := \depth(\mathcal{T}_\Idx) - L_\ell$.

Define the \emph{parabolically admissible leaves} corresponding to $(\widehat \sigma,\widehat \tau)$ by 
\begin{equation*}
\lfpa(\widehat \sigma,\widehat \tau):= {\mathcal T}^{L_{\ell}^\text{middle}}_\Idx(\widehat \sigma) 
\times {\mathcal T}^{L_{\ell}^\text{middle}}_\Idy(\widehat \tau).
\end{equation*}
\item
Define the set of admissible leaves by
\begin{equation*}
\lfaIJ:= \bigcup_{(\widehat \sigma,\widehat \tau) \in \lfsaIJ}  \lfpa(\widehat \sigma,\widehat \tau). 
\end{equation*}
\end{enumerate}

In order to approximate $\mathbf{K}$, we consider each block
$(\widehat{\sigma},\widehat{\tau})\in\mathcal{L}^{standard}_{\Idx\times\Idy}$
individually:
if it is an inadmissible block, we store
$\mathbf{K}|_{\widehat{\sigma}\times\widehat{\tau}}$ directly.
If it is an admissible block, we apply the butterfly representation
described in the previous section to the
sub-clustertrees $\mathcal{T}_\Idx(\sigma)$ and
$\mathcal{T}_\Idy(\tau)$.
This is equivalent to approximating
$\mathbf{K}|_{\widehat{\sigma}\times\widehat{\tau}}$ by a \emph{local}
$\mathcal{DH}^2$-matrix.

%
%
\begin{remark}
In Section~\ref{sec:DH2} we argued that a matrix with a butterfly structure
can be understood as a ${\mathcal D}{\mathcal H}^2$-matrix. The situation is different 
here, where only submatrices are endowed with a butterfly structure. While 
the submatrices are ${\mathcal D}{\mathcal H}^2$-matrices, the \emph{global} matrix is not a 
${\mathcal D}{\mathcal H}^2$-matrix. 
To see this, let $p:=\depth(\mathcal{T}_\Idx)$.
If we start with an admissible block $\sigma\times\tau$ (with respect to the standard admissibility
condition) on
level $\ell$, choose a middle level
\begin{equation*}
  L_\ell^\text{middle} = p - L_\ell
  = p - \lfloor (p-\ell)/2 \rfloor
  = \ell + (p-\ell) - \lfloor (p-\ell)/2 \rfloor
  = \ell + \lceil (p-\ell)/2 \rceil, 
\end{equation*}
and re-interpolate $p-\ell^\text{middle}$ times until we reach leaf
clusters $\sigma_p,\tau_p$, we will use points
$x_{\sigma_{\hat\ell}}$ and $y_{\tau_{\hat\ell}}$ on level
\begin{align*}
  \hat\ell
  &:= L_\ell^\text{middle} - (p - \ell^\text{middle})
   = 2 L_\ell^\text{middle} - p
   = 2 (\ell + \lceil (p-\ell)/2 \rceil) - p
\\ &
= 2 \ell + 2 \lceil (p-\ell)/2 \rceil - p
   = \begin{cases}
       \ell &\text{ if } p-\ell \text{ is even},\\
       \ell+1 &\text{ otherwise}
     \end{cases}
\end{align*}
for the approximation, i.e., the point sets $\mathcal{D}_\sigma$ and
$\mathcal{D}_\tau$ depend on the level $\ell$ of the admissible block.
For a $\mathcal{DH}^2$-matrix, these sets are only allowed to depend
on $\sigma$ and $\tau$, but not on the level $\ell$ of the admissible
block.
\eremk
\end{remark}

The error analysis of the resulting matrix approximation will require some assumptions. The following assumptions
will be useful in Section~\ref{sec:helmholtz} for the analysis of the
Helmholtz kernel (\ref{eq:helmholtz-kernel}). 

%
%
\begin{assumption}
\label{assumption:model-situation}
\begin{enumerate}
\item Blocks $(\widehat \sigma,\widehat\tau) \in \lfsaIJ$ satisfy the admissibility condition (\ref{eq:eta_1}). 
\item For blocks $(\widehat \sigma,\widehat\tau) \in \lfsaIJ$ there holds for all $\widehat{\ell} \in \{L_\ell^\text{middle}-L_\ell,\ldots,L_\ell^\text{middle}+L_\ell\}$
\begin{equation}
\label{eq:eta_2}
\kappa \operatorname*{diam} B_{\sigma'} 
\operatorname*{diam} B_{\tau'} \leq \eta_2 \operatorname*{dist}(B_{\widehat \sigma}, B_{\widehat \tau})
\quad \text{ for all } \sigma' \in {\mathcal T}^{L_\ell^{\widehat{\ell}}}_\Idx(\widehat \sigma),\ \tau' \in {\mathcal T}^{L_\ell^{\widehat{\ell}}}_\Idy(\widehat \tau)
\end{equation}
for some fixed parameter $\eta_2$. 
\item 
(shrinking condition) 
There is a constant $\overline{q}\in(0,1)$ such that
for all blocks $(\widehat{\sigma},\widehat{\tau}) \in \lfsaIJ$ there
holds for all $\sigma \in {\mathcal T}_\Idx(\widehat{\sigma})$,
$\sigma' \in \sons(\sigma)$ as well as all
$\tau \in {\mathcal T}_\Idy(\widehat{\tau})$, $\tau' \in \sons(\tau)$ 
\begin{equation}
\label{eq:q} 
\diam_i B_{\sigma'} \leq 
\overline{q} \diam_i B_{\sigma},
\quad 
\diam_i B_{\tau'} \leq 
\overline{q} \diam_i B_{\tau} \quad\text{ for all } i=1,\ldots,d.  
\end{equation}
\end{enumerate}
\end{assumption}

\section{Analysis}
\label{sec:analysis}

A key point of the error analysis is the understanding of the re-interpolation step, which 
hinges on the following question: Given, on an interval $[-1,1]$, an analytic function 
that is the product of an analytic function and a polynomial, how well can 
we approximate it by polynomials on a subinterval $[a,b] \subset [-1,1]$? In turn, sharp estimates for 
polynomial approximation of analytic functions rely on bounds of the function to be approximated 
on Bernstein's elliptic discs. The following Lemma~\ref{lemma:inclusion}, which is a refinement 
of \cite[Lemma~{5.4}]{boerm-melenk15}, shows that, given $\rho_1>1$, it is possible to ensure 
$\ellipse^{[a,b]}_{\rho_0} \subset \ellipse^{[-1,1]}_{\rho_1}$ in conjunction with 
 $\rho_0 > \rho_1$: 
%
%
\begin{lemma}[Inclusion]
\label{lemma:inclusion}
Fix $\overline{q}  \in (0,1)$ and $\underline{\rho}> 1$. Then there exists $\widehat q \in (0,1)$ such that for 
any $\rho_0 \ge \underline{\rho}$ there exists $\rho_1 \in (1,\widehat q \rho_0]$ such that 
for any interval $[a,b] \subset [-1,1]$ with $(b-a)/2 \leq \overline{q}$ there holds 
\begin{equation}
\label{eq:lemma:inclusion-100}
{\ellipse}^{[a,b]}_{\rho_0} \subset {\ellipse}_{\rho_1}. 
\end{equation}
In fact, the smallest $\rho_1$ satisfying (\ref{eq:lemma:inclusion-100}) is given by 
the solution $\rho_1 > 1$ of the quadratic equation (\ref{eq:lemma:inclusion-1000}).  
\end{lemma}
\begin{proof} 
We remark that \cite[Lemma~{5.4}]{boerm-melenk15} represents
a simplified version of Lemma~\ref{lemma:inclusion} that is suitable for large $\rho_0$; 
in particular, for $\rho_0 \rightarrow \infty$, the 
ratio $\rho_1/\rho_0$ tends to $(b-a)/2$. The proof of the present general case
is relegated to \ref{sec:appendix}.
\end{proof}

In view of our assumption (\ref{eq:q}), we can fix a ``shrinking factor''
$\overline{q}\in(0,1)$ in the following.
We also assume $\underline{\rho}>1$; the parameter $\widehat{q}\in(0,1)$ appearing in the following
results will be as in Lemma~\ref{lemma:inclusion}.

We study one step of re-interpolation in the univariate case: 

%
%
\begin{lemma}
\label{lemma:univariate-case}
Let $J_0 = [a_0,b_0]$ and $J_1 = [a_1,b_1]$ be two intervals with
$J_1 \subset J_0$.
Let $x_1\in J_1$.
Set $h_0:= (b_0-a_0)/2$ and $h_1 = (b_1-a_1)/2$ and assume 
\begin{equation*}
h_1/h_0 \leq \overline{q} < 1. 
\end{equation*}
Let $G$ be holomorphic on ${\ellipse}^{[a_1,b_1]}_{\rho_1}$ for some
$\rho_1 \ge \underline{\rho} > 1$. 
Assume that $G|_{J_1}$ is real-valued. Let $\kappa \ge 0$.
Then there exists $\widehat q \in (0,1)$ depending solely on $\overline{q}$ and $\underline{\rho}$ such that 
\begin{gather}
\inf_{v \in {\mathcal P}_m} \|\exp(\bi \kappa (\punkt-x_1) G) \pi - v\|_{L^\infty(J_1)} 
 \leq C_G \widehat q^m \|\pi\|_{L^\infty(J_0)} 
\qquad \text{ for all } \pi \in {\mathcal P}_m,\label{eq:lemma:univariate-case-5}\\
C_G := \frac{2}{\rho_1-1}
\exp\left(\kappa h_1 \left(\frac{\rho_1+1/\rho_1}{2}+1\right)\|G\|_{L^\infty({\ellipse}^{[a_1,b_1]}_{\rho_1})}\right)\notag. 
\end{gather}
Hence, for the interpolation error we get 
\begin{align}
\label{eq:lemma:univariate-case-10}
\|\exp(\bi \kappa (\punkt - x_1) G) \pi - I_m( \exp(\bi \kappa (\punkt - x_1) G) \pi)\|_{L^\infty(a_1,b_1)}
\leq (1 + \Lambda_m) C_G \widehat q^m \|\pi\|_{L^\infty(J_0)}. 
\end{align}
\end{lemma}
\begin{proof}
Let $x_m := (b_1+a_1)/2$ denote the midpoint of $J_1$.
Since $x_1 \in J_1$, we have $|x - x_1| \leq |x-x_m| + |x_m-x_1|
\leq ((\rho_1+1/\rho_1)/2+1)h_1$ for
any $x \in {\ellipse}^{[a_1,b_1]}_{\rho_1} \subseteq B_{(\rho_1+1/\rho_1)/2}(x_m)$.  
We estimate (generously) with the abbreviations $ M:= \|G\|_{L^\infty({\ellipse}^{[a_1,b_1]}_{\rho_1})}$
and $R:= (\rho_1+1/\rho_1)/2+1$
\begin{equation*}
\left| \operatorname*{Im} (x - x_1) G(x) \right| \leq |x - x_1|\,|G(x)| \leq h_1 R M 
\qquad \text{ for all } x \in {\ellipse}^{[a_1,b_1]}_{\rho_1}. 
\end{equation*}
We conclude  with a polynomial approximation result (cf.~\cite[eqn. (8.7) in proof of Thm.~8.1, Chap.~7]{DELO93})
\begin{equation}
\label{eq:exponential-polynomial-approximation}
\inf_{v \in {\mathcal P}_m} \|\exp\left(\bi \kappa (\punkt-x_1) G\right) \pi - v\|_{L^\infty(J_1)} \leq 
\frac{2}{\rho_1-1} \exp(\kappa h_1 M R) \rho_1^{-m} \|\pi\|_{L^\infty({\ellipse}^{[a_1,b_1]}_{\rho_1})}. 
\end{equation}
Let $\rho_0$ be the smallest value such that ${\ellipse}^{[a_1,b_1]}_{\rho_1} \subset {\ellipse}^{[a_0,b_0]}_{\rho_0}$
as given by Lemma~\ref{lemma:inclusion}.
In particular, we have $\rho_1/\rho_0 \leq \widehat q$ with $\widehat q$ given by Lemma~\ref{lemma:inclusion}.
By Bernstein's estimate (\cite[Chap.~4, Thm.~{2.2}]{DELO93}) 
we can estimate
$$
\|\pi\|_{L^\infty({\ellipse}^{[a_1,b_1]}_{\rho_1})} \leq 
\|\pi\|_{L^\infty({\ellipse}^{[a_0,b_0]}_{\rho_0})} \leq \rho_0^m \|\pi\|_{L^\infty(J_0)}
$$
and arrive at 
$$
\inf_{v \in {\mathcal P}_m} \|\exp\left(\bi \kappa (\punkt-x_1) G\right) \pi - v\|_{L^\infty(J_1)} \leq 
\frac{2}{\rho_1-1} \exp(\kappa h_1 M R) \rho_1^{-m} \rho_0^m \|\pi\|_{L^\infty(J_0)} . 
$$
Recalling $\rho_1/\rho_0 \leq \widehat q$ allows us to finish the proof of 
(\ref{eq:lemma:univariate-case-5}).  The estimate 
(\ref{eq:lemma:univariate-case-10}) then follows from Lemma~\ref{lemma:polynomial-interpolation-error}. 
\end{proof}
\begin{remark}
The limiting case $\rho_1 \rightarrow \infty$ corresponds to an entire bounded $G$, which is, 
by Liouville's theorem a constant. This particular case is covered in \cite[Lemma~{5.4}]{boerm-melenk15}.  
\eremk
\end{remark}
It is convenient to introduce, for $z \in \BbbR^d$, the function $E_{z}$ and the operator $\widehat {\mathfrak I}^B_{z}$ by 
\begin{align}
x \mapsto E_{z}(x) &:= \exp \left( \bi \kappa \phase(x,z) \right),\\ 
\widehat {\mathfrak I}^{B}_{z} f &:= E_{z} I^B_m \left( \frac{1}{E_{z}} f\right). 
\end{align}

The multivariate version of Lemma~\ref{lemma:univariate-case} is as follows: 

%
%
\begin{lemma}
\label{lemma:re-interpolation-one-step}
Let $\overline{q} \in (0,1)$ and $\underline{\rho} > 1$.
Let $B_1^x  := [\bs{a}_{1}^x, \bs{b}_{1}^x] \subset B_0^x:=[\bs{a}_{0}^x,\bs{b}_{0}^x] 
\subset \BbbR^d$ and 
$B_{0}^y := [\bs{a}_{0}^y,\bs{b}_{0}^y] 
\subset B_{-1}^y := [\bs{a}_{-1}^y,\bs{b}_{-1}^y]\subset \BbbR^d$. 
Let $y_0 \in B_0^y$, $y_{-1} \in B_{-1}^y$ be given. 
Assume
\begin{equation*}
(b_{1,i}^x - a_{1,i}^x) \leq \overline{q} (b_{0,i}^x - a_{0,i}^x)
    \quad\text{ for all } i =1,\ldots,d.
\end{equation*}
Assume furthermore for a $\rho \ge \underline{\rho}$ and an open set
$\Omega \subset \BbbC^{2d}$ with 
${\ellipse}^{[\bs{a}^x_0, \bs{b}^x_0]}_{{\rho}} \times {\ellipse}^{[\bs{a}^y_{-1}, \bs{b}^y_{-1}]}_{{\rho}} 
\subset \Omega$ that the function $\phase \in L^\infty(\Omega)$ is analytic on $\Omega$. Set 
\begin{equation*}
d_{\Omega}:= \sup \{\varepsilon>0\,|\, B_\varepsilon(x) \times B_\varepsilon(y) \subset \Omega
\text{ for all } (x,y) \in 
{\ellipse}^{[\bs{a}^x_0, \bs{b}^x_0]}_{{\rho}} \times {\ellipse}^{[\bs{a}^y_{-1}, \bs{b}^y_{-1}]}_{{\rho}} 
\}.   
\end{equation*}
Assume, for some $\gamma> 0$,  
\begin{equation}
\kappa \|y_0 - y_{-1}\| 
\diam B^X_1 
\frac{\|\phase\|_{L^\infty(\Omega)}}{d_\Omega^2} 
\leq \gamma.  
\end{equation}
Then there holds for a $\widehat{q} \in (0,1)$ that depends solely on $\overline{q}$ and $\underline{\rho}$
\begin{gather}
\|E_{y_{0}} \pi  - 
\widehat{\mathfrak I}^{B^x_1}_{y_{-1}} (E_{y_{0}} \pi) \|_{L^\infty(B_1^x)} 
\leq C_T \widehat q^m \|\pi\|_{L^\infty(B_0^x)}  
\qquad \text{ for all } \pi \in {\mathcal Q}_m,  
\label{eq:lemma:re-interpolation-one-step-100}\\
C_T := \frac{2d}{\rho-1} (1+\Lambda_m)^d 
\exp\left(\gamma \left(\frac{\rho + 1/\rho}{2}+1\right)\right). 
\label{eq:lemma:re-interpolation-one-step-200}
\end{gather}
\end{lemma}
\begin{proof}
Fix $x_1 \in B_1^x$ and $x_0 \in B_0^x$. Notice that with the function $R_{x_1,y_{-1}}$ of Lemma~\ref{lemma:G} 
\begin{align*}
\frac{E_{y_{0}}(x)}{E_{y_{-1}}(x)} & = \exp\left(\bi \kappa \left( R_{x_1,y_{-1}}(x,y_{0}) + \phase(x_1,y_{0}) 
-\phase(x_1,y_{-1})\right)\right)  
\\ & 
= \exp\left(\bi \kappa R_{x_1,y_{-1}}(x,y_{0})\right) \exp\left(\bi \kappa (\phase(x_1,y_{0}) - \phase(x_1,y_{-1}))\right). 
\end{align*}
By Lemma~\ref{lemma:G}, we have $R_{x_1,y_{-1}}(x,y_{0}) = (x - x_1)^\top G(x,y_{0}) (y_{0} - y_{-1})$ for 
a function $G$ that is analytic on $\Omega$ and satisfies 
\begin{equation*}
\sup_{(x,y) \in {\ellipse}^{[\bs{a}^x_0,\bs{b}^x_0]}_{{\rho}} \times {\ellipse}^{[\bs{a}^y_{-1},\bs{y}^y_{-1}]}_{{\rho}}}
|G(x,y)| \leq \frac{\|\phase\|_{L^\infty(\Omega)}}{d_\Omega^2}. 
\end{equation*}
Noting that $\exp(\bi \kappa (\phase(x_1,y_{0})-\phase(x_1,y_{-1})))$ is a constant of modulus $1$, we
obtain by combining 
the univariate interpolation estimate (\ref{eq:lemma:univariate-case-10}) of 
Lemma~\ref{lemma:univariate-case} with the multivariate interpolation error estimate 
(\ref{eq:error-estimate-tensor-product-interpolation}) 
the bound 
\begin{align*}
\|E_{y_{0}} \pi - \widehat{\mathfrak I}^{B^X_1}_{y_{-1}} (E_{y_{0}} \pi)  \|_{L^\infty(B_1^x)} 
\leq \frac{2d}{\rho-1} (1+\Lambda_m)^d \exp\left(\gamma \left(\frac{\rho + 1/\rho}{2}+1\right)\right) \widehat q^m \|\pi\|_{L^\infty(B_0^x)}. 
\end{align*}
\end{proof}

Lemma~\ref{lemma:re-interpolation-one-step} handles one step of a
re-interpolation process. The following theorem studies the question of
iterated re-interpolation: 

%
%
\begin{theorem}
\label{thm:nested-interpolation}
Let $\overline{q} \in (0,1)$ and $\underline{\rho} > 1$.
Let $B_\ell^x = [\bs{a}_\ell^x,\bs{b}_\ell^x]$,
$\ell=0,\ldots,L$, be a nested sequence with $B_L^x \subset B_{L-1}^x \subset \cdots \subset B_0^x$ 
satisfying the shrinking condition 
\begin{equation}
\label{eq:thm:nested-interpolation-10}
(b_{\ell+1,i}^x - a_{\ell+1,i}^x) \leq \overline{q} (b_{\ell,i}^x - a_{\ell,i}^x), 
\quad\text{ for all } i =1,\ldots,d \text{ and } \ell=0,\ldots,L-1.
\end{equation}
Let $B^y = [\bs{a}^y,\bs{b}^y]\subset \BbbR^d$.
Assume that $\phase$ is analytic on the open set $\Omega \subset \BbbC^{2d}$
with $\ellipse^{[\bs{a}_0^x,\bs{b}_0^x]}_\rho \times
\ellipse^{[\bs{a}^y,\bs{b}^y]}_\rho \subset \Omega$ for some $\rho \ge \underline{\rho}>1$. 
Define 
\begin{equation}
\label{eq:thm:nested-interpolation-dOmega}
d_\Omega:= \sup\{\varepsilon > 0\,|\, B_\varepsilon(x) \times B_\varepsilon(y) \subset \Omega
\text{ for all } (x,y) \in 
\ellipse^{[\bs{a}_0^x,\bs{b}_0^x]}_\rho \times \ellipse^{[\bs{a}^y,\bs{b}^y]}_\rho\}.
\end{equation}
Let $(y_{-i})_{i=0}^{L} \subset B^y$ be a sequence of points.  
Assume furthermore 
\begin{equation}
\label{eq:thm:nested-interpolation-15}
\kappa \diam B^{[\bs{a}^x_{\ell},\bs{b}^x_{\ell}]} \|y_{-\ell-1}-y_{-\ell}\| 
\frac{\|\phase\|_{L^\infty(\Omega)}}{d_\Omega^2} 
\leq \gamma 
\qquad \text{ for all } \ell = 0,\ldots,L. 
\end{equation}
Abbreviate the operators 
\begin{equation*}
{\mathfrak I}_\ell:= \widehat {\mathfrak I}^{B^X_{\ell}}_{y_{-\ell}}, \qquad \ell=0,\ldots,L. 
\end{equation*}
Then, for a $\widehat q \in (0,1)$ depending solely on $\overline{q}$ and $\underline{\rho}$ and 
with the constant
\begin{equation}
\label{eq:thm:nested-interpolation-C1}
C_1:= \frac{2d}{\rho-1} (1 + \Lambda_m)^d \exp\left(\gamma \left(\frac{\rho+1/\rho}{2}+1\right)\right)
\end{equation}
there holds for $\ell=1,\ldots,L$: 
\begin{align} 
\label{eq:thm:nested-interpolation-20}
\|(\operatorname*{I} -  {\mathfrak I}_\ell \circ \cdots \circ {\mathfrak I}_1) (E_{y_0}\pi)  \|_{L^\infty(B^x_{\ell})} 
&\leq ((1 + C_1 \widehat q^m)^{\ell} - 1)  \|\pi\|_{L^\infty(B^x_0)} 
\quad \text{ for all } \pi \in {\mathcal Q}_m, \\
\label{eq:thm:nested-interpolation-25}
 \|({\mathfrak I}_\ell \circ \cdots \circ {\mathfrak I}_1) (E_{y_0}\pi)  \|_{L^\infty(B^x_{\ell})} 
 &  \leq (1 + C_1 \widehat q^m)^{\ell}  \|\pi\|_{L^\infty(B^x_0)} %
\quad \text{ for all }  \pi \in {\mathcal Q}_m, %
\\
\label{eq:thm:nested-interpolation-30}
\| {\mathfrak I}_\ell \circ \cdots \circ {\mathfrak I}_0 \|_{C(B^x_\ell)\leftarrow C(B^x_0)} &\leq 
\Lambda_m^d ( 1 + C_1 \widehat q^m)^{\ell}. 
\end{align}
\end{theorem}
\begin{proof}
{\em Step~1:} By Lemma~\ref{lemma:re-interpolation-one-step}, we have the following approximation 
property for the operators ${\mathfrak I}_\ell$:
with $C_1$ given by
(\ref{eq:thm:nested-interpolation-C1}) (cf. the definition of $C_T$ in 
(\ref{eq:lemma:re-interpolation-one-step-200})) we find
\begin{align}
\label{eq:thm:nested-interpolation-50}
\|(\operatorname*{I} - {\mathfrak I}_\ell) (E_{y_{-\ell+1}} \pi) \|_{L^\infty(B_{\ell}^x)} 
\leq C_1 \widehat q^m \|\pi\|_{L^\infty(B_{\ell-1}^x)} 
= C_1 \widehat q^m \|E_{y_{-\ell+1}} \pi\|_{L^\infty(B_{\ell-1}^x)}, 
\end{align}
where the last equality follows from the fact that $\phase(x,z)$ is real
for real arguments $x$ and $z$, so we have $|E_{y_{-\ell+1}}|=1$.

{\em Step~2:}
Observe the telescoping sum 
\begin{align}
\label{eq:thm:nested-interpolation-80}
\widetilde E_\ell &:= \operatorname*{I} - {\mathfrak I}_\ell \circ \cdots \circ {\mathfrak I}_1 
= 
(\operatorname*{I} - {\mathfrak I}_1) + 
(\operatorname*{I} - {\mathfrak I}_2) \circ {\mathfrak I}_1  + 
(\operatorname*{I} - {\mathfrak I}_3) \circ {\mathfrak I}_2 \circ {\mathfrak I}_1  + \cdots + 
(\operatorname*{I} - {\mathfrak I}_{\ell}) \circ {\mathfrak I}_{\ell-1} \circ \cdots \circ {\mathfrak I}_1.   
\end{align}
We claim the following estimates: 
\begin{align} 
\label{eq:thm:nested-interpolation-100}
\|\widetilde E_\ell (E_{y_{0}} \pi) \|_{L^\infty(B^x_{\ell})} &\leq ((1+ C_1\widehat q^m)^{\ell}-1) \|\pi\|_{L^\infty(B^x_0)}, \\ 
\label{eq:thm:nested-interpolation-200}
\|{\mathfrak I}_\ell \circ \cdots \circ {\mathfrak I}_1 (E_{y_0} \pi) \|_{L^\infty(B^x_{\ell})} 
&\leq (1+ C_1 \widehat q^m)^{\ell} \|\pi\|_{L^\infty(B^x_0)}. 
\end{align}
This is proven by induction on $\ell$. For $\ell = 1$, the estimate (\ref{eq:thm:nested-interpolation-100}) 
expresses (\ref{eq:thm:nested-interpolation-50}), and 
(\ref{eq:thm:nested-interpolation-200}) follows with an additional application of the triangle inequality
since ${\mathfrak I}_1 = \operatorname*{I} - \widetilde E_1$. 
The case $\ell = 0$ is trivial as ${\mathfrak I}_\ell \circ \cdots \circ {\mathfrak I}_1$ is understood 
as the identity.
To complete the induction argument, 
assume that there is an $n\in\mathbb{N}$ such that
(\ref{eq:thm:nested-interpolation-100}),
  (\ref{eq:thm:nested-interpolation-200}) hold for all
$\ell\in\{0,\ldots,\min\{n,L-1\}\}$.
Let $\ell\in\{0,\ldots,\min\{n,L-1\}\}$ and $\pi\in\mathcal{Q}_m$.
We observe that there is a $\widetilde{\pi}\in\mathcal{Q}_m$ such that
${\mathfrak I}_\ell \circ \cdots \circ {\mathfrak I}_1 (E_{y_0} \pi) = 
E_{y_{-\ell}} \widetilde{\pi}$.
The induction hypothesis and (\ref{eq:thm:nested-interpolation-200}) imply
\begin{align}
\nonumber 
 \| (\operatorname*{I} - {\mathfrak I}_{\ell+1}) {\mathfrak I}_\ell \circ \cdots \circ {\mathfrak I}_1 (E_{y_0} \pi)\|_{L^\infty(B_{\ell+1}^x)}  
&= \| (\operatorname*{I} - {\mathfrak I}_{\ell+1}) (E_{y_{-\ell}} \widetilde \pi)\|_{L^\infty(B_{\ell+1}^x)}    \\
 \qquad \stackrel{(\ref{eq:thm:nested-interpolation-50})}{\leq} C_1 \widehat q^m \|E_{y_{-\ell}} \widetilde \pi\|_{L^\infty(B_\ell^x)} 
\label{eq:thm:nested-interpolation-2000} 
 & = C_1 \widehat q^m \|
({\mathfrak I}_\ell \circ \cdots \circ {\mathfrak I}_1) (E_{y_0} \pi) \|_{L^\infty(B_\ell^x)}. 
\end{align}
Now let $\ell=\min\{n,L-1\}$.
We get from 
(\ref{eq:thm:nested-interpolation-80}), 
(\ref{eq:thm:nested-interpolation-200}), 
(\ref{eq:thm:nested-interpolation-2000}),
and the geometric series 
\begin{align*}
\|\widetilde E_{\ell+1} (E_{y_0} \pi)\|_{L^\infty(B^x_{\ell+1})} &
\leq \sum_{i=0}^\ell
\|(\operatorname*{I} - {\mathfrak I}_{i+1}) ({\mathfrak I}_i \circ \cdots \circ {\mathfrak I}_1) E_{y_0} \pi\|_{L^\infty(B^x_{\ell+1})} 
\\ & 
\stackrel{(\ref{eq:thm:nested-interpolation-2000})}{\leq}
\sum_{i=0}^\ell C_1 \widehat q^m 
\|({\mathfrak I}_i \circ \cdots \circ {\mathfrak I}_1) E_{y_0} \pi\|_{L^\infty(B^x_i)}
\stackrel{(\ref{eq:thm:nested-interpolation-200})}{\leq}
\sum_{i=0}^\ell C_1 \widehat q^m
 (1 + C_1 \widehat q^m)^i\|\pi\|_{L^\infty(B^x_0)} 
\\ & 
= C_1 \widehat q^m \|\pi\|_{L^\infty(B^x_0)} \frac{(1+ C_1 \widehat q^m)^{\ell+1} - 1}{(1+C_1 \widehat q^m) - 1}
= ((1 + C_1 \widehat q^m)^{\ell+1} - 1) \|\pi\|_{L^\infty(B^x_0)}, 
\end{align*}
%
%
which is the desired induction step for (\ref{eq:thm:nested-interpolation-100}). The induction step 
for (\ref{eq:thm:nested-interpolation-200}) now follows with the triangle inequality. 

{\em Step 3:} 
The estimate (\ref{eq:thm:nested-interpolation-100}) is the desired estimate 
(\ref{eq:thm:nested-interpolation-20}). The bound (\ref{eq:thm:nested-interpolation-30}) follows 
from (\ref{eq:thm:nested-interpolation-200}) and the stability properties of $I^{B^X_0}_m$: For $u \in C(B^X_0)$
we compute 
\begin{align*} 
\| {\mathfrak I}_\ell \circ \cdots \circ {\mathfrak I}_0 u\|_{C(B^x_\ell)} &= 
\| {\mathfrak I}_\ell \circ \cdots \circ {\mathfrak I}_1 ({\mathfrak I}_0 u)\|_{C(B^x_\ell)}  = 
\| {\mathfrak I}_\ell \circ \cdots \circ {\mathfrak I}_1 (E_{y_0} I^{B^X_0}_m u)\|_{C(B^x_\ell)}   \\
&\stackrel{(\ref{eq:thm:nested-interpolation-200})}{\leq} 
( 1 + C_1 \widehat q^m)^{\ell}  \|I^{B^X_0}_m u\|_{C(B^X_0)}
\leq 
( 1 + C_1 \widehat q^m)^{\ell}  \Lambda_m^d \|u\|_{C(B^X_0)}. 
\qedhere
\end{align*}
\end{proof}
We are now in a position to prove our main result, namely, an error estimate for the butterfly approximation
of the kernel function $k$ given in (\ref{eq:kernel}):

%
%
\begin{theorem}[Butterfly approximation by interpolation]
\label{thm:main}
Let $B_L^X \subset B_{L-1}^X \subset \cdots \subset B_{-L}^X$
and $B_L^Y \subset B_{L-1}^Y \subset \cdots \subset B_{-L}^Y$
be two sequences of the form 
$B_\ell^X  = [\bs{a}_{\ell}^x, \bs{b}_{\ell}^x]$, 
$B_\ell^Y  = [\bs{a}_{\ell}^y, \bs{b}_{\ell}^y]$. 
Let $(x_\ell)_{\ell=-L}^L$, $(y_\ell)_{\ell=-L}^L$ be two sequences with 
$x_\ell \in B^X_\ell$ and 
$y_\ell \in B^Y_\ell$, $\ell=-L,\ldots,L$. 
\emph{Assume:}
\begin{itemize}
\item (analyticity of $\phase$ and $\amp$)
Let ${\rho}_\phase > 1$ and ${\rho}_\amp > 1$ be such that 
$\amp$ is holomorphic on 
${\ellipse}^{[\bs{a}_{0}^x,\bs{b}_{0}^x]}_{\rho_\amp} \times {\ellipse}^{[\bs{a}_{0}^y,\bs{b}_{0}^y]}_{\rho_\amp}$
and the phase function 
$\phase$ is holomorphic on 
$\Omega \supset {\ellipse}^{[\bs{a}_{-L}^x,\bs{b}_{-L}^x]}_{\rho_\phase} \times {\ellipse}^{[\bs{a}_{-L}^y,\bs{b}_{-L}^y]}_{\rho_\phase} 
\supset 
{\ellipse}^{[\bs{a}_{0}^x,\bs{b}_{0}^x]}_{\rho_\amp} \times {\ellipse}^{[\bs{a}_{0}^y,\bs{b}_{0}^y]}_{\rho_\amp}$. 
Define 
\begin{align}
\label{eq:thm:main-120}
M_\amp& := \|\amp\|_{L^\infty(\ellipse^{[\bs{a}_0^x,\bs{b}_0^x]}_{\rho_\amp} \times \ellipse^{[\bs{a}_0^y,\bs{b}_0^y]}_{\rho_\amp})},  \\
\label{eq:thm:main-130}
M_\phase&:= \|\phase\|_{L^\infty(\Omega)}, \\
d_\Omega &:= \sup \{ \varepsilon >0\,|\, B_\varepsilon(x) \times
B_\varepsilon(y) \subset \Omega
\notag \\ &
\qquad\qquad \text{ for all } (x,y) \in 
\ellipse^{[\bs{a}_{0}^x,\bs{b}_{0}^x]}_{\rho_\amp} \times \ellipse^{[\bs{a}_{-L}^y,\bs{b}_{-L}^y]}_{\rho_\phase} 
\cup 
\ellipse^{[\bs{a}_{-L}^x,\bs{b}_{-L}^x]}_{\rho_\phase} \times \ellipse^{[\bs{a}_{0}^y,\bs{b}_{0}^y]}_{\rho_\amp} 
\}. \label{eq:thm:main-135}
\end{align}
\item (shrinking condition)
Let $\overline{q} \in (0,1)$ be such that 
\begin{subequations}
\label{eq:thm:main-shrinking} 
\begin{align}
\label{eq:thm:main-shrinking-a} 
(b_{\ell+1,i}^x - a_{\ell+1,i}^x) \leq \overline q 
(b_{\ell,i}^x - a_{\ell,i}^x), 
\qquad i=1,\ldots,d, \quad \ell=-L,\ldots,L-1, \\  
\label{eq:thm:main-shrinking-b} 
(b_{\ell+1,i}^y - a_{\ell+1,i}^y) \leq \overline q 
(b_{\ell,i}^y - a_{\ell,i}^y), 
\qquad i=1,\ldots,d, \quad \ell=-L,\ldots,L-1. 
\end{align}
\end{subequations}
\item Let $\gamma > 0$ be such that 
\begin{subequations}
\label{eq:thm:main-butterfly-condition} 
\begin{align}
\label{eq:thm:main-butterfly-condition-a} 
\kappa \operatorname*{diam} B^X_\ell \|y_{-\ell} - y_{-\ell-1}\| \frac{M_\phase}{d_\Omega^2} & \leq \gamma, \qquad \ell=0,\ldots,L-1, \\
\label{eq:thm:main-butterfly-condition-b} 
\kappa \operatorname*{diam} B^Y_\ell \|x_{-\ell} - x_{-\ell-1}\| \frac{M_\phase}{d_\Omega^2} & \leq \gamma, \qquad \ell=0,\ldots,L-1, \\ 
\label{eq:thm:main-butterfly-condition-c} 
\kappa \operatorname*{diam} B^X_0 \operatorname*{diam} B^Y_0  \frac{M_\phase}{d_\Omega^2} & \leq \gamma. 
\end{align}
\end{subequations}
\item (polynomial growth of Lebesgue constant)
Let $C_\Lambda>0$ and $\lambda > 0$ be such that the Lebesgue constant of the underlying interpolation
process satisfies 
\begin{equation}
\label{eq:polynomial-growth-Lambda}
\Lambda_m \leq C_\Lambda (m+1)^\lambda \qquad \text{ for all } m \in \BbbN_0. 
\end{equation}
\end{itemize}
\emph{Then:} There exist constants $C$, $b$, $b'> 0$ depending only on 
$\gamma$, $\rho_\amp$, $\rho_\phase$, $d$, $C_\Lambda$, $\lambda$ such that under the constraint
\begin{equation}
\label{eq:thm:main-side-constraint} 
m \ge b' \log(L+2)
\end{equation}
the following approximation result holds for $k$ given by (\ref{eq:kernel}): 
\begin{align*}
& \|k - 
\left( {\mathfrak I}^{B_L^X,x}_{y_{-L}} \circ \cdots \circ {\mathfrak I}^{B_0^X,x}_{y_0} \right) \otimes 
\left( {\mathfrak I}^{B_L^Y,y}_{x_{-L}} \circ \cdots \circ {\mathfrak I}^{B_0^Y,y}_{x_0} \right)  k
\|_{L^\infty(B^X_L \times B^Y_L)} \leq 
 C \exp(-b m) M_\amp. 
\end{align*}
\end{theorem}
\begin{proof}
%
\emph{Step 1:} We may assume $\rho_\phase \ge \rho_\amp$. 
It is convenient to abbreviate
\begin{equation*}
\widehat{\mathfrak I}^x_\ell:= {\mathfrak I}^{B^X_\ell,x}_{y_{-\ell}}, 
\qquad 
\widehat{\mathfrak I}^y_\ell:= {\mathfrak I}^{B^Y_\ell,y}_{x_{-\ell}}. 
\end{equation*}
We note 
\begin{align}
\label{eq:H}
H &:= 
(\widehat{\mathfrak I}^x_0 \otimes \widehat{\mathfrak I}^y_0)[ k](x,y) \\
\nonumber 
&= \exp(\bi \kappa \phase(x,y_0)) \exp(\bi \kappa \phase(x_0,y))  
I^{B^X_0}_m \otimes I^{B^Y_0}_m \underbrace{ \left[\exp(\bi \kappa (R_{x_0,y_0}(x,y) - \phase(x_0,y_0))) \amp(x,y)\right]}_{=: F(x,y)}, 
\end{align}
where the function $R_{x_0,y_0}$ is defined in Lemma~\ref{lemma:G}. Using the representation of $R_{x_0,y_0}$ given there, 
we write 
\begin{equation*}
F(x,y)=  \exp(- \bi \kappa \phase(x_0,y_0)) \amp(x,y) \exp(\bi \kappa ( (x-x_0)^\top G(x,y) (y - y_0)),  
\end{equation*}
where the function $G$ is holomorphic on the domain $\Omega$. We estimate 

\begin{align*}
\sup_{(x,y) \in \ellipse^{B^X_0} _{\rho_\amp} \times {B^Y_0}} \kappa | (x - x_0)^\top G(x,y) (y - y_0)| & \leq 
\kappa \operatorname*{diam} B^X_0 \frac{\rho_\amp + 1/\rho_\amp}{2} 
\|G\|_{L^\infty(\ellipse^{B^X_0}_{\rho_\amp} \times B^Y_0)} 
\operatorname*{diam} B^Y_0  
 \stackrel{(\ref{eq:thm:main-butterfly-condition-c})} {\leq }
\gamma \frac{\rho_\amp+1/\rho_\amp}{2},  
\end{align*}
and get with analogous arguments 
\begin{align*}
\sup_{(x,y) \in {B^X_0}  \times \ellipse^{B^Y_0}_{\rho_\amp}} \kappa | (x - x_0)^\top G(x,y) (y - y_0)| & \leq 
\gamma \frac{\rho_\amp+1/\rho_\amp}{2}.  
\end{align*}
Lemma~\ref{lemma:polynomial-interpolation-error} in conjunction
with Lemma~\ref{lemma:G} implies together with the univariate polynomial approximation result 
that led to (\ref{eq:exponential-polynomial-approximation}) 
\begin{align}
\label{eq:F}
\|F - I^{B^X_0}_m \otimes I^{B^Y_0}_m F\|_{L^\infty(B^X_0 \times B^Y_0)} &\leq \frac{4d}{\rho_A-1} (1+\Lambda_m)^{2d} 
\exp(\gamma (\rho_\amp + 1/\rho_\amp)/2) M_\amp \rho_A^{-m}. 
\end{align}
Recall the definition of $H$ in (\ref{eq:H}). 
Since $\phase$ is real for real arguments, the estimate (\ref{eq:F}) yields 
\begin{align}
\label{eq:estimate-with-H} 
\|k - H\|_{L^\infty(B^X_0 \times B^Y_0)} &\leq \frac{4d}{\rho_A-1} (1+\Lambda_m)^{2d} 
\exp(\gamma (\rho_\amp + 1/\rho_\amp)/2) M_\amp \rho_A^{-m}. 
\end{align}
\emph{Step 2:} We quantify the effect of 
$(\widehat{\mathfrak I}^x_L \circ \cdots \circ \widehat{\mathfrak I}^x_1) \otimes (\widehat{\mathfrak I}^y_L \circ \cdots \circ \widehat{\mathfrak I}^y_1)$. 
The key is to observe that Theorem~\ref{thm:nested-interpolation} is applicable since this operator is applied to the function 
$H$, 
which is a tensor product of functions of the form suitable for an application of 
Theorem~\ref{thm:nested-interpolation}. We note with the constants $C_1$,
$\widehat q \in (0,1)$ of Theorem~\ref{thm:nested-interpolation}
\begin{align}
\nonumber 
& \|(\operatorname*{I} - 
(\widehat {\mathfrak I}^x_L \circ \cdots \circ \widehat{\mathfrak I}^x_1) 
\otimes (\widehat {\mathfrak I}^y_L \circ \cdots \circ \widehat{\mathfrak I}^y_1)) H\|_{L^\infty(B^X_L \times B^Y_L)}  \\
\nonumber 
& \leq  
\|(\operatorname*{I} - 
(\widehat {\mathfrak I}^x_L \circ \cdots \circ \widehat{\mathfrak I}^x_1)
\otimes \operatorname*{I}) H \|_{L^\infty(B^X_L \times B^Y_L)}
+ 
\|
(\widehat {\mathfrak I}^x_L \circ \cdots \circ \widehat{\mathfrak I}^x_1)\otimes 
\left( \operatorname*{I} - 
(\widehat {\mathfrak I}^y_L \circ \cdots \circ \widehat{\mathfrak I}^y_1)\right) H\|_{L^\infty(B^X_L \times B^Y_L)}  \\
\nonumber 
&\leq (1 + (1 + C_1 \widehat q^m)^L) \left( (1 + C_1 \widehat q^m)^L - 1\right) \|H\|_{L^\infty(B^X_0 \times B^Y_0)}  \\ 
\label{eq:foo}
&\leq \underbrace{(1 + (1 + C_1 \widehat q^m)^L)}_{=:C_{m,L}}  \underbrace{ \left( (1 + C_1 \widehat q^m)^L - 1\right) }_{=:\widehat \varepsilon_{m,L}} 
\left(\|k\|_{L^\infty(B^X_0 \times B^Y_0)} + \|k - H\|_{L^\infty(B^X_0 \times B^Y_0)}
\right).
\end{align}
We get, noting the trivial bound $\|k\|_{L^\infty(B^X_0 \times B^Y_0)} \leq M_\amp$,  
\begin{align} 
\nonumber 
& \|k - \left( (\widehat {\mathfrak I}^x_L \circ \cdots \circ \widehat{\mathfrak I}^x_1) 
\otimes (\widehat {\mathfrak I}^y_L \circ \cdots \circ \widehat{\mathfrak I}^y_1) \right)(\widehat {\mathfrak I}^x_0 \otimes \widehat {\mathfrak I}^y_0) k
\|_{L^\infty(B^X_L \times B^X_L)} \\
\nonumber 
& 
\leq \|k - H\|_{L^\infty(B^X_L \times B^Y_L)} + \|\left(\operatorname*{I} - 
(\widehat {\mathfrak I}^x_L \circ \cdots \circ \widehat{\mathfrak I}^x_1) 
\otimes (\widehat {\mathfrak I}^y_L \circ \cdots \circ \widehat{\mathfrak I}^y_1) \right)H \|_{L^\infty(B^X_L \times B^Y_L)} \\ 
\nonumber 
& \stackrel{(\ref{eq:foo})}{\leq }
\|k - H\|_{L^\infty(B^X_0 \times B^Y_0)} + 
C_{m,L} \widehat \varepsilon_{m,L}  \left( M_\amp + \|k - H \|_{L^\infty(B^X_0 \times B^Y_0)}\right)\\
\label{eq:thm:main-1000}
& 
\stackrel{(\ref{eq:estimate-with-H})}{\leq}
\frac{4d}{\rho_\amp-  1} (1 + \Lambda_m)^{2d} \exp(\gamma (\rho_A + 1/\rho_A)/2) M_\amp \rho_\amp^{-m} 
\left(1 + C_{m,L} \widehat \varepsilon_{m,L}\right) 
+ 
C_{m,L} \widehat \varepsilon_{m,L}  M_\amp. 
\end{align}
\emph{Step~3:}
(\ref{eq:thm:main-1000}) is valid for arbitrary $m$ and $L$.  We simplify 
(\ref{eq:thm:main-1000}) by making a further assumption on the relation between $m$ and $L$:
The assumption (\ref{eq:polynomial-growth-Lambda}) on $\Lambda_m$ implies that for any 
chosen $\widetilde q \in (\widehat q,1)$ we have for sufficiently large $m$
\begin{equation*}
  (1 + \Lambda_m)^d \widehat q^m
  \leq (1 + C_\Lambda (m+1)^\lambda)^d \widehat q^m
  \leq \widetilde q^m.
\end{equation*}
Hence, we obtain for a suitable constant $C > 0$ that is independent of $m$
\begin{equation*}
  \widehat \varepsilon_{m,L} = (1 +C_1 \widehat q^m)^L-1 
  \leq (1 + C\widetilde q^m)^{L}-1 
  \stackrel{1+x \leq e^x}{\leq} \exp(C \widetilde q^m L) - 1. 
\end{equation*}
Using the estimate $\exp(x) - 1 \leq e x$, which is valid for $x \in [0,1]$,
and assuming that $C \widetilde q^m L \leq 1$ (note that this holds for
$m \ge K \log (L+2)$  for sufficiently large $K$), we obtain
\begin{align*}
\widehat \varepsilon_{m,L} \leq C e \widetilde q^m L = 
     C e \exp( m \ln (\widetilde q) +  \ln L) 
     \leq C e \exp( m \ln (\widetilde q) + m/K) \leq C' \exp (-b m), 
\end{align*}
where $b > 0$ if we assume that $K$ is selected sufficiently large. Inserting
this estimate in (\ref{eq:thm:main-1000}) and noting that $C_{m,L} = 2+ \widehat \varepsilon_{m,L}$ 
allows us to conclude the proof. 
\end{proof}
\section{Application: the 3D Helmholtz kernel}
\label{sec:helmholtz}
The case of the 3D Helmholtz kernel
\begin{equation}
\label{eq:helmholtz-kernel}
k_{Helm}(x,y) = \frac{\exp(\bi \kappa \|x - y\|)}{4 \pi  \|x - y\|}
\end{equation}
corresponds to the phase function $\phase(x,y) = \|x - y\|$ and the 
amplitude function $\amp(x,y) = 1/(4\pi\|x - y\|)$. We illustrate the 
butterfly representation for a Galerkin discretization of the 
single layer operator, i.e., 
\begin{equation*}
  \varphi \mapsto (V \varphi)(x)
  := \int_{y \in \Gamma} k_{Helm}(x,y) \varphi(y) \,dy,
\end{equation*}
where $\Gamma$ is a bounded surface in $\BbbR^3$. Given a family
of shape functions $(\varphi_i)_{i=1}^N$, the stiffness matrix ${\mathbf K}$
is given by 
\begin{equation}
\label{eq:stiffness-matrix-helmholtz}
{\mathbf K}_{i,j} = \int_{x \in\Gamma} \int_{y \in \Gamma} k_{Helm}(x,y)
        \varphi_j(y)\varphi_i(x) \,dy\,dx. 
\end{equation}
We place ourselves in the setting of Section~\ref{sec:butterfly-asymptotically-smooth-kernels}
with $\Idx = \Idy = \{1,\ldots,N\}$. 

%
%
\begin{theorem}
\label{thm:helmholtz}
Assume the setting of Section~\ref{sec:butterfly-asymptotically-smooth-kernels} and let 
Assumption~\ref{assumption:model-situation} be valid. 
Then there are constants $C$, $b$, $b'>0$ that depend solely on the 
admissibility parameters $\eta_1$, $\eta_2$,  and 
the parameter $q \in (0,1)$ of Assumption~\ref{assumption:model-situation} 
such that for the stiffness matrix 
${\mathbf K} \in \BbbC^{\Idx \times \Idx} $ given by (\ref{eq:stiffness-matrix-helmholtz}) and its approximation 
$\widetilde {\mathbf K} \in \BbbC^{\Idx \times \Idx}$ that is obtained by the butterfly representation 
as described in Section~\ref{sec:butterfly-asymptotically-smooth-kernels} 
the following holds: If $m \ge b' \log (2 + \depth{\mathcal T}_\Idx)$ then 
\begin{align*}
\sup_{(i,j) \in \widehat \sigma\times \widehat \tau} |{\mathbf K}_{i,j} - \widetilde{\mathbf K}_{i,j}| 
\leq  C \frac{\|\varphi_i\|_{L^1(\Gamma)} \|\varphi_j\|_{L^1(\Gamma)}}{\operatorname*{dist}(B_{\widehat \sigma},B_{\widehat \tau})}
\begin{cases}
\exp(-b m) & \mbox{ if $(\widehat \sigma,\widehat \tau) \in \lfsaII$} \\
0  & \mbox{ if $(\widehat \sigma,\widehat \tau) \in \lfsiII$}.  
\end{cases}
\end{align*}
\end{theorem}
\begin{proof}
We apply Theorem~\ref{thm:main} for blocks $(\widehat \sigma,\widehat\tau) \in \lfsaII$. 
To that end, we note that Lemma~\ref{lemma:admissibility-complex-norm} gives
us the existence
of $\varepsilon > 0$ and $\rho> 1$ (depending only on the admissibility parameter $\eta_1$) such 
that phase function $(x,y) \mapsto \phase(x,y) = \|x - y\|$ is holomorphic on 
\begin{equation*}
  \Omega := \bigcup \{
    B_{\varepsilon \delta_{\widehat{\sigma}\widehat{\tau}}}(x) \times
    B_{\varepsilon \delta_{\widehat{\sigma}\widehat{\tau}}}(y) \,|\, 
   (x,y) \in \ellipse^{B_{\widehat \sigma}}_\rho\times\ellipse^{B_{\widehat \tau}}_\rho\},
   \qquad
   \delta_{\widehat{\sigma}\widehat{\tau}}
   := \operatorname*{dist} (B_{\widehat \sigma}, B_{\widehat \tau}),
\end{equation*}
and satisfies
\begin{align*}
  \sup_{(x,y) \in \Omega} |\phase(x,y)|
  &\leq C \delta_{\widehat{\sigma}\widehat{\tau}}, &
  \inf_{(x,y) \in \Omega} |\phase(x,y)|
  &\geq C^{-1} \delta_{\widehat{\sigma}\widehat{\tau}}.
\end{align*}
Hence, the constants $M_\phase$, $M_\amp$, and $d_\Omega$, $\rho_\amp$, $\rho_\phase$ 
appearing in Theorem~\ref{thm:main} can be bounded by 
\begin{align*}
  M_\phase
  &\lesssim \delta_{\widehat{\sigma}\widehat{\tau}}, &
  M_\amp
  &\lesssim 1/\delta_{\widehat{\sigma}\widehat{\tau}}, &
  d_\Omega
  &\gtrsim \delta_{\widehat{\sigma}\widehat{\tau}}, &
  \rho_\amp
  &= \rho_\phase = \rho.
\end{align*}
We observe 
\begin{equation*}
  \frac{M_\phase}{d_\Omega^2}
  \lesssim \frac{1}{\delta_{\widehat{\sigma}\widehat{\tau}}}
\end{equation*}
so that the conditions (\ref{eq:thm:main-butterfly-condition}) of Theorem~\ref{thm:main}  are 
satisfied in view of our Assumption in (\ref{eq:eta_2}). The result now follows from Theorem~\ref{thm:main}. 
\end{proof}
We conclude this section with a proof of the fact that the Euclidean norm
admits a holomorphic extension.

%
%
\begin{lemma}
\label{lemma:holomorphic-extension-euclidean-norm}
Let $\omega \subset \BbbR^d$ be open. Define the set 
\begin{equation}
\label{eq:cone}
{\mathcal C}_\omega:= \bigcup_{x \in \omega} B_{(\sqrt{2} - 1) |x|}(x) \subset \BbbC^d .
\end{equation}
Then the function 

\begin{equation*}
  \enorm \colon \omega \to \BbbC,\qquad
         x \mapsto \sqrt{\sum_{i=1}^d x_i^2}, 
\end{equation*}

has an analytic extension to ${\mathcal C}_\omega$. 
Furthermore, 
\begin{equation}
\label{eq:lemma:holomorphic-extension-euclidean-norm-100}
\sqrt{\left| \operatorname*{Re} \sum_{i=1}^d z_i^2\right|} 
\leq 
\sqrt{\left|\sum_{i=1}^d z_i^2\right|}  = |\enorm(z_1,\ldots,z_d) | = 
\sqrt{\left|\sum_{i=1}^d z_i^2\right|} \leq \sqrt{\sum_{i=1}^d |z_i|^2}.
\end{equation}
\end{lemma}
\begin{proof} 
The assertion of analyticity will follow from Hartogs' theorem (cf., e.g., \cite[Thm.~{2.2.8}]{hoermander90}), 
which states that a function that is analytic in each variable separately is in fact analytic. 
In order to apply Hartogs' theorem, we ascertain that ${\mathcal C}_\omega$ is chosen in such a way that 
\begin{equation*}
  \operatorname*{Re} \sum_{i=1}^d z_i^2 > 0 
  \qquad \text{ for all } (z_1,\ldots,z_d) \in {\mathcal C}_\omega.
\end{equation*}
To see this, abbreviate $D:= \sqrt{2} - 1$ and 
write $(z_1,\ldots,z_d) \in {\mathcal C}_\omega$ in the form $z_i = x_i + \zeta_i$ with $x \in \omega$ 
and $\zeta_i \in \BbbC$ with $\sum_{i=1}^d|\zeta_i|^2 < D^2 |x|^2$. Then, with Young's inequality with 
$\delta := D = \sqrt{2}-1$:
\begin{align*}
\operatorname*{Re} \sum_{i=1}^d z_i^2 &= 
\operatorname*{Re} \sum_{i=1}^d (x_i+\zeta_i) ^2 \ge 
\sum_{i=1}^d \left(x_i^2 - 2 |x_i| |\zeta_i| - |\zeta_i|^2\right)
\ge 
\|x\|^2 - \delta \|x\|^2 - \delta^{-1} \|\zeta\|^2 - \|\zeta\|^2\\
&> (1 - \delta - D^2/\delta - D^2)|x|^2  = 
(1 - 2 D - D^2)\|x\|^2 \stackrel{D = \sqrt{2}-1}{=} 0.  
\end{align*}
Since the square root function is well-defined on the right half plane 
$\{z \in \BbbC\,|\, \operatorname*{Re} z > 0\}$, the function $\enorm$ is naturally defined on ${\mathcal C}_\omega$ and 
analytic in each variable separately. 
The equalities $|\enorm(z_1,\ldots,z_d)| = \sqrt{|\sum_{i=1}^d z_i^2|}$ in 
(\ref{eq:lemma:holomorphic-extension-euclidean-norm-100}) follow from the equation
$|\sqrt{z}| = \sqrt{|z|}$ for $z \in \BbbC$ with $\operatorname*{Re} z > 0$, 
and the two inequalities in (\ref{eq:lemma:holomorphic-extension-euclidean-norm-100}) are straightforward. 
\end{proof}

%
%
\begin{lemma}
\label{lemma:admissibility-complex-norm}
Let $\eta > 0$. Then there exist $\varepsilon>0$ and  $\rho > 1$ depending solely on $\eta$ and the spatial dimension $d$ 
such that the following is true for any  
${[\bs{a}^x,\bs{b}^x]}$ and ${[\bs{a}^y,\bs{b}^y]}$ satisfying the admissibility condition 
\begin{equation}
\label{eq:lemma:admissibility-complex-norm-10}
\eta \operatorname*{dist}({[\bs{a}^x,\bs{b}^x]}, {[\bs{a}^y,\bs{b}^y]}) \ge 
\max\{\operatorname*{diam}( {[\bs{a}^x,\bs{b}^x]}), \operatorname*{diam} ({[\bs{a}^y,\bs{b}^y]}) \}. 
\end{equation}
Set $\delta_B:= \operatorname*{dist}({[\bs{a}^x,\bs{b}^x]},{[\bs{a}^y,\bs{b}^y]})$
and define
\begin{equation*}
  \Omega := \bigcup \{B_{\varepsilon \delta_B}(x) \times
                      B_{\varepsilon \delta_B}(y)\,|\, 
      (x,y) \in \ellipse^{[\bs{a}^x,\bs{b}^x]}_\rho \times
                \ellipse^{[\bs{a}^y,\bs{b}^y]}_\rho\} \subset \BbbC^{2d}.
\end{equation*}
Then the function $(x,y) \mapsto \|x - y\|$ has an analytic extension $(x,y) \mapsto \enorm(x-y)$ on
$\Omega$, and this extension satisfies, 
for a constant $C > 0$ that also depends solely on $\eta$ and $d$, 
\begin{align}
\label{eq:lemma:admissibility-complex-norm-20}
\sup_{(x,y) \in \Omega} 
|\enorm(x-y)| & \leq C \operatorname*{dist}({[\bs{a}^x,\bs{b}^x]}, {[\bs{a}^y,\bs{b}^y]}),   \\
\label{eq:lemma:admissibility-complex-norm-30}
\inf_{(x,y) \in \Omega}
|\enorm(x-y)| & \ge C^{-1} \operatorname*{dist}({[\bs{a}^x,\bs{b}^x]}, {[\bs{a}^y,\bs{b}^y]}). 
\end{align}
\end{lemma}
\begin{proof} 
It is convenient to introduce the abbreviations 
\begin{gather*}
D:= \max\{\operatorname*{diam}( {[\bs{a}^x,\bs{b}^x]}),
          \operatorname*{diam} ({[\bs{a}^y,\bs{b}^y]}) \}, 
\\
\Omega_x := \bigcup \{ B_{\varepsilon \delta_B}(x) \,|\,
                       x \in \ellipse^{[\bs{a}^x,\bs{b}^x]}_\rho\},
\qquad
\Omega_y := \bigcup \{ B_{\varepsilon \delta_B}(y) \,|\,
                       y \in \ellipse^{[\bs{a}^y,\bs{b}^y]}_\rho\}.
\end{gather*}
We identify $\operatorname*{Re} \Omega_x$ and 
$\operatorname*{Re} \Omega_y$. 
We start by observing that
$\operatorname*{Re} {\ellipse}_{\bs{\rho}}
 = {\ellipse}_{\bs{\rho}} \cap \BbbR^d
 = B^{[\bs{a},\bs{b}]}$ with 
\begin{equation*}
a_i = - \frac{1}{2} \left(\bs{\rho}_i + \frac{1}{\bs{\rho}_i}\right), 
\quad 
b_i = \frac{1}{2} \left(\bs{\rho}_i + \frac{1}{\bs{\rho}_i}\right)
\qquad\text{ for all } i=1,\ldots,d. 
\end{equation*}
More generally, $\operatorname*{Re} {\ellipse}^{[\bs{a},\bs{b}]}_{\bs{\rho}} = 
{\ellipse}^{[\bs{a},\bs{b}]}_{\bs{\rho}}  \cap \BbbR^d$ 
is again a box obtained from the box 
${[\bs{a},\bs{b}]}$ by stretching the $i$-th direction by a factor
$1/2 (\bs{\rho}_i + 1/\bs{\rho}_i)$.
We now restrict to the case that $\bs{\rho}_i = \rho$ for all $i=1,\ldots,d$. 
We note that 
\begin{align}
\nonumber 
\operatorname*{dist}( \operatorname*{Re} \ellipse^{[\bs{a},\bs{b}]}_\rho,[\bs{a},\bs{b}])   & = 
\operatorname*{dist}( \ellipse^{[\bs{a},\bs{b}]}_\rho \cap {\mathbb R}^d,[\bs{a},\bs{b}])  
\leq 
\sqrt{d} \left(\frac{\rho+1/\rho}{2} - 1\right) \max_{i=1,\ldots,d} (b_i - a_i)  \\
\label{eq:foobar}
&\leq 
\sqrt{d} \left(\frac{\rho+1/\rho}{2} - 1\right) \operatorname*{diam} ([\bs{a},\bs{b}]). 
\end{align}
Using (\ref{eq:foobar}) and a triangle inequality, we obtain from (\ref{eq:lemma:admissibility-complex-norm-10}) 
for $\rho > 1$ sufficiently small 
\begin{align*}
 \operatorname*{dist}(\operatorname*{Re}\Omega_x, \operatorname*{Re} \Omega_y) & \ge 
 \operatorname*{dist}({[\bs{a}^x,\bs{b}^x]}, {[\bs{a}^y,\bs{b}^y]}) - 
2 \sqrt{d} \left(\frac{\rho+1/\rho}{2}-1\right) D  - 2 \sqrt{d} \varepsilon \delta_B 
\\ & 
\ge \left(1-2\sqrt{d} \varepsilon  - {2 \sqrt{d} } \eta \left(\frac{\rho + 1/\rho}{2} - 1\right)\right)  \delta_B. 
\end{align*}
Consider now the set 
$$
\omega:= \{x - y\,|\, x \in \operatorname*{Re}\Omega_x, \ y \in \operatorname*{Re}\Omega_y\}
$$
and ${\mathcal C}_\omega$ as defined by (\ref{eq:cone}). 
Note that for $z \in \Omega_x$ we have 
\begin{align*}
\operatorname*{Re} z &\in \operatorname*{Re}\Omega_x,\\
\left| \operatorname*{Im} z_i \right|
&\leq \frac{\rho - 1/\rho}{2} D + \varepsilon \delta_B
 \leq \left(\eta \frac{\rho - 1/\rho}{2} + \varepsilon\right)\delta_B
\qquad\text{ for all } i=1,\ldots,d,
\end{align*}
with an analogous statement about 
$\zeta \in \Omega_y$. We conclude for 
$z \in \Omega_x$ and 
$\zeta \in \Omega_y$ that the difference 
\begin{align*}
z - \zeta   = \underbrace{\operatorname*{Re} (z - \zeta)}_{=:\alpha \in \omega} + \bi \operatorname*{Im} (z - \zeta) 
\end{align*}
satisfies 
\begin{align*}
\|\alpha\| &\ge 
\operatorname*{dist}(\operatorname*{Re}\Omega_x,\operatorname*{Re}\Omega_y) 
\ge 
\left(1 - 2 \sqrt{d} \varepsilon - 2 \sqrt{d} \eta \left(\frac{\rho + 1/\rho}{2}-1\right)\right) \delta_B ,  \\
\sum_{i=1}^d |\operatorname*{Im} (z_i - \zeta_i)|^2 
& \leq 2 \sum_{i=1}^d |\operatorname*{Im} z_i|^2 + |\operatorname*{Im} \zeta_i|^2 
\leq 4 d \left( \eta \frac{\rho-1/\rho}{2} + \varepsilon\right)^2 \delta_B^2. 
\end{align*}
Hence, $z-\zeta \in {\mathcal C}_\omega$ provided 
\begin{align*}
  \frac{4 d \left(\eta(\rho-1/\rho)/2+\varepsilon\right)^2 \delta_B^2}
       {\|\alpha\|^2}
  &\leq \frac{4d(\eta(\rho-1/\rho)/2+\epsilon)^2}
             {1-2 \sqrt{d} \varepsilon-2\sqrt{d}\eta ((\rho+1/\rho)/2-1)}
   \leq (\sqrt{2} - 1)^2. 
\end{align*}
Selecting first $\varepsilon$ sufficient small and then $\rho$ sufficiently close to $1$, this last condition 
can be ensured.  
By Lemma~\ref{lemma:holomorphic-extension-euclidean-norm}, we conclude the desired analyticity assertion
as well as the upper bound (\ref{eq:lemma:admissibility-complex-norm-20}) on $|\enorm(x-y)|$. For the lower 
bound (\ref{eq:lemma:admissibility-complex-norm-30}), we use 
\begin{align*}
|\enorm(z-\zeta)| &\ge \sqrt{\left| \operatorname*{Re} \sum_{i=1}^d (z_i - \zeta_i)^2\right|} = 
\sqrt{\left|\sum_{i=1}^d (\operatorname*{Re} (z_i - \zeta_i))^2 - (\operatorname*{Im} (z_i - \zeta_i))^2\right|}
\\
& \ge \sqrt{\|\alpha\|^2 - 4 d \left(\eta(\rho-1/\rho)/2+\varepsilon\right)^2 \delta_B^2}
\ge C \delta_B, 
\end{align*}
where $C > 0$ depends only on $\eta$ and $d$. 
\end{proof}

%
%
\begin{remark}
 
The condition (\ref{eq:eta_2}) can be met. To illustrate this, assume that the basis functions 
$(\varphi_i)_{i=1}^N$ all have support of size $O(h)$. For each index $i$, fix a ``proxy''
$x_i \in \operatorname*{supp} \varphi_i$, e.g., the barycenter of $\operatorname*{supp} \varphi_i$.  
Consider a tree $\mathcal{T}_\Idx$ for the index set $\Idx=\{1,\ldots,N\}$ that results from organizing 
the proxies $(x_i)_{i=1}^N$ in a tree based on a standard octree. In particular, to each 
cluster $\sigma \in \mathcal{T}_\Idx$ we can associate a box $B^{oct}_{\sigma}$ of the octree such that 
$$
i \in \sigma \quad \Longleftrightarrow\quad x_i \in B^{oct}_{\sigma}. 
$$
The tree $\mathcal{T}_\Idx$, which was created using the proxies, is also a cluster tree for the shape 
functions $\varphi_i$. 
The bounding box $B_\sigma\supset B^{oct}_\sigma$ for $\sigma$ can be chosen close to $B^{oct}_\sigma$ 
in the sense that 
$$
\operatorname*{diam}_i B^{oct}_\sigma \leq 
\operatorname*{diam}_i B_\sigma \leq 
\operatorname*{diam}_i B^{oct}_\sigma + C h. 
$$
This allows us to show that (\ref{eq:eta_2}) and (\ref{eq:q}) can be met if the leaf size is sufficiently large: 
For $\sigma$ and $\sigma' \in \sons(\sigma)$ we compute
\begin{align*}
\frac{\operatorname*{diam}_i B_{\sigma'}} 
     {\operatorname*{diam}_i B_{\sigma}} 
\leq 
\frac{\operatorname*{diam}_i B^{oct}_{\sigma'}+ C h} 
     {\operatorname*{diam}_i B^{oct}_{\sigma} } = 
\frac{1}{2} + C 
\frac{h} 
     {\operatorname*{diam}_i B^{oct}_{\sigma} }.
\end{align*}
This last expression can be made $<1$ if the leaves are not too small, i.e., if the smallest boxes 
of the octree are large compared to $h$. 
Let us consider (\ref{eq:eta_2}). We assume that for 
$\sigma \in {\mathcal T}^{L^\text{middle}_\ell+i}_\Idx (\widehat \sigma)$ and 
$\tau\in {\mathcal T}^{L^\text{middle}_\ell-i}_\Idx (\widehat \tau)$ we have 
$$
\kappa \operatorname*{diam} B^{oct}_\sigma \operatorname*{diam}B^{oct}_\tau 
\lesssim \operatorname*{dist} (B^{oct}_{\widehat \sigma},B^{oct}_{\widehat \tau}) 
\sim 
\operatorname*{dist} (B_{\widehat \sigma},B_{\widehat \tau}). 
$$
Furthermore, 
we observe $h \lesssim \operatorname*{diam} B_{\widehat \sigma}^{oct}$ and 
$h \lesssim \operatorname*{diam} B_{\widehat \tau}^{oct}$ so that 
we can estimate 
\begin{align*}
\kappa \operatorname*{diam} B_\sigma \operatorname*{diam}B_\tau 
&\leq 
\kappa (\operatorname*{diam} B^{oct}_\sigma + Ch) (\operatorname*{diam}B^{oct}_\tau +Ch) \\
&\leq 
\kappa \operatorname*{diam} B^{oct}_\sigma \operatorname*{diam}B^{oct}_\tau 
+ 2\kappa h \max\{\operatorname*{diam} B^{oct}_{\sigma}, \operatorname*{diam} B^{oct}_\tau\} + C \kappa h h 
\lesssim \operatorname*{dist}(B_{\widehat \sigma}, B_{\widehat \tau}). 
\end{align*}
\eremk
\end{remark}

\section{Numerical experiment}
\label{sec:numerics}

In view of the main result of Theorem~\ref{thm:main}, we expect the
butterfly approximation $\widetilde{\mathbf{K}}$ to converge exponentially
to $\mathbf{K}$ as the degree $m$ of the interpolation polynomials is increased.

In order to get an impression of the convergence properties of the
approximation scheme, we apply the butterfly approximation to the 
discretization (\ref{eq:stiffness-matrix-helmholtz}) of the Helmholtz 
single layer operator. The surface $\Gamma$ is taken to be the polyhdral 
approximation of the unit sphere $\{ x\in\BbbR^3\ :\ \|x\|_2=1 \}$ that is obtained
by applying regular refinement to the
sides of the double pyramid $\{ x\in\BbbR^3\ :\ \|x\|_1=1 \}$ and
projecting the resulting vertices to the sphere. These polyhedra form quasi-uniform
meshes. The test and trial spaces consist of piecewise constant functions 
on these meshes, taking the characteristic functions of the elements are the basis functions 
$\varphi_i$. 
When forming the stiffness matrix the singular integrals are evaluated by
the quadratures described in \cite{ERSA98,SASC11}, while regular integrals
are evaluated by tensor quadrature in combination with the
Duffy transformation.

The cluster tree $\mathcal{T}_\Idx$ is constructed by finding an
axis-parallel bounding box containing the entire surface $\Gamma$
and bisecting it simultaneously in all coordinate directions.
Due to this construction, clusters can have at most eight sons
(empty boxes are discarded) and the diameters of the son boxes
are approximately half the diameter of their father.
The subdivision algorithm stops on the first level containing
a cluster with not more than $32$ indices.
The block tree is constructed by the standard admissibility
condition using the parameter $\eta_1=1$.

%
%
\begin{table}
  \begin{equation*}
    \begin{array}{r|rr|rr|rr}
        & \multicolumn{2}{c|}{n=32\,768,\kappa=16}
        & \multicolumn{2}{c|}{n=73\,728,\kappa=24}
        & \multicolumn{2}{c}{n=131\,072,\kappa=32}\\
      m & \|\mathbf{K}-\widetilde{\mathbf{K}}\|_2
        & \text{factor}
        & \|\mathbf{K}-\widetilde{\mathbf{K}}\|_2
        & \text{factor}
        & \|\mathbf{K}-\widetilde{\mathbf{K}}\|_2
        & \text{factor}\\
      \hline
      0 & 3.09_{-6\phantom{0}} &
        & 1.35_{-6\phantom{0}} & 
        & 8.69_{-7\phantom{0}} & \\
      1 & 5.60_{-7\phantom{0}} & 5.51
        & 2.05_{-7\phantom{0}} & 6.58
        & 2.37_{-7\phantom{0}} & 3.67\\
      2 & 5.02_{-8\phantom{0}} & 11.16
        & 1.28_{-8\phantom{0}} & 15.99
        & 2.46_{-8\phantom{0}} & 9.61\\
      3 & 4.64_{-9\phantom{0}} & 10.82
        & 1.10_{-9\phantom{0}} & 11.63
        & 2.68_{-9\phantom{0}} & 9.20\\
      4 & 4.61_{-10} & 10.06
        & 9.80_{-11} & 11.26
        & 3.20_{-10} & 8.38
    \end{array}
  \end{equation*}

  \caption{Estimated spectral errors for the butterfly approximation}
  \label{table:spectral_slp}
\end{table}

%
%
\begin{table}
  \begin{equation*}
    \begin{array}{r|rr|rr|rr}
        & \multicolumn{2}{c|}{n=32\,768,\kappa=16}
        & \multicolumn{2}{c|}{n=73\,728,\kappa=24}
        & \multicolumn{2}{c}{n=131\,072,\kappa=32}\\
      m & \|\mathbf{K}-\widetilde{\mathbf{K}}\|_F
        & \text{factor}
        & \|\mathbf{K}-\widetilde{\mathbf{K}}\|_F
        & \text{factor}
        & \|\mathbf{K}-\widetilde{\mathbf{K}}\|_F
        & \text{factor}\\
      \hline
      0 & 3.57_{-5\phantom{0}} &
        & 2.08_{-5\phantom{0}} & 
        & 1.32_{-5\phantom{0}} & \\
      1 & 3.82_{-6\phantom{0}} & 9.37
        & 2.31_{-6\phantom{0}} & 9.00
        & 1.85_{-6\phantom{0}} & 7.17\\
      2 & 2.91_{-7\phantom{0}} & 13.13
        & 1.39_{-7\phantom{0}} & 16.62
        & 1.41_{-7\phantom{0}} & 13.06\\
      3 & 2.47_{-8\phantom{0}} & 11.77
        & 9.82_{-9\phantom{0}} & 14.17
        & 1.31_{-8\phantom{0}} & 10.80\\
      4 & 2.40_{-9\phantom{0}} & 10.29
        & 7.86_{-10} & 12.49
        & 1.32_{-9\phantom{0}} & 9.91
    \end{array}
  \end{equation*}

  \caption{Frobenius errors for the butterfly approximation}
  \label{table:frobenius_slp}
\end{table}

The butterfly approximation is constructed by tensor product Chebyshev
interpolation.
Table~\ref{table:spectral_slp} lists the spectral errors for
$n\in\{32\,768,73\,728,131\,072\}$ triangles with wave numbers
$\kappa\in\{16,24,32\}$, corresponding to $\kappa h \approx 0.6$,
i.e., approximately ten mesh elements per wavelength.
The spectral errors are estimated by applying the power iteration
to approximate the spectral radius of the Gramian of the error.
We can see that the error reduction factors are quite stable
and close to $10$.
Table~\ref{table:frobenius_slp} lists the error in the Frobenius
norm.
The Frobenius error is computed by direct comparison with the
exact matrix $\mathbf{K}$.
We observe convergence at a rate close to $10$.

\appendix 
\section{Proofs of auxiliary results}
\label{sec:appendix} 
\begin{numberedproof}{of Lemma~\ref{lemma:inclusion}}

We consider $h \in (0,1]$ and the ellipse ${\ellipse}^{[-1,-1+2h]}_{\rho_0}$. We note that once we find $\rho_1$ such that 
${\ellipse}_{\rho_1} \supset {\ellipse}^{[-1,-1+2h]}_{\rho_0}$, then by symmetry we also have 
${\ellipse}_{\rho_1} \supset {\ellipse}^{[1,1-2h]}_{\rho_0}$ and then, by convexity of ${\ellipse}_{\rho_1}$ also 
${\ellipse}_{\rho_1} \supset {\ellipse}^{[x-h,x+h]}_{\rho_0}$ for any $x \in [-1+h,1-h]$. 
This justifies our restricting to ${\ellipse}^{[-1,-1+2h]}_{\rho_0}$. 
In Cartesian coordinates, this ellipse is given by 
\begin{gather*}
\left( \frac{x - (-1 + h)}{a}\right)^2 + \left( \frac{y}{b}\right)^2 = 1,
\qquad
a = h \frac{\rho_0 + 1/\rho_0}{2}, 
\qquad
b = h \frac{\rho_0 - 1/\rho_0}{2},
\end{gather*}
and $x \in [-1+h-a,-1+h+a]$. 
The value $\rho_1>1$ such that 
${\ellipse}^{[-1, -1+2h]}_{\rho_0} \subset {\ellipse}_{\rho_1}$ has to satisfy 
\begin{equation}
\label{eq:lemma:inclusion-10}
\sup_{x \in [-1+h-a,-1+h+a]} \sqrt{(x-1)^2 + y(x)^2} + \sqrt{(x+1)^2 + y(x)^2} =: \widetilde M \stackrel{!}{\leq} \rho_1 + 1/\rho_1, 
\end{equation}
where 
$$
y(x)^2 = b^2 - \frac{b^2}{a^2} (x - (-1+h))^2. 
$$
{\em Claim:} The maximum value $\widetilde M$ in (\ref{eq:lemma:inclusion-10}) is attained at the left 
endpoint $x = -1+h-a$ and is given by 
\begin{equation}
\widetilde M = 2 r_x \left( \frac{1}{r_x} + h \left( 1 - \frac{1}{r_x}\right)\right), 
\qquad r_x:= \frac{1}{2} \left(\rho_0 + \frac{1}{\rho_0}\right). 
\end{equation} 
To compute the supremum in (\ref{eq:lemma:inclusion-10}) we introduce the function 
$f:= s_1 + s_2$ with 
\begin{align*}
s_1(x) &= \sqrt{(x-1)^2 + b^2 - (b/a)^2 (x+1-h)^2}, & 
s_2(x) &\strut = \sqrt{(x+1)^2 + b^2 - (b/a)^2 (x+1-h)^2}.
\end{align*}
The special structure of the values $a$ and $b$ implies that $s_2$ is actually a polynomial: 
\begin{equation*}
s_2(x) = \sqrt{1 - \frac{b^2}{a^2}} (x+1) + \sqrt{b^2 - \frac{b^2}{a^2} h^2}. 
\end{equation*}
We compute $f$ at the endpoints: 
\begin{subequations}
\label{eq:f-endpoints}
\begin{align*}
f(-1+h-a) &= 2 + a - h + \sqrt{1-b^2/a^2} (h-a) + \sqrt{b^2 - h^2 b^2/a^2}, \\
\nonumber 
f(-1+h+a) &= \sqrt{(a - 2 +h)^2} + \sqrt{1-b^2/a^2}(h+a) + \sqrt{b^2 - h^2 b^2/a^2}. 
\end{align*}
For the left endpoint, a direct calculation yields 
\begin{equation}
\label{eq:f-endpoints-a}
\frac{f(-1+h-a)}{2 r_x} = \frac{1}{r_x} + h \left(1 - \frac{1}{r_x}\right). 
\end{equation}
For the right endpoint, we get similar, simplified formulas, distinguishing the caes 
$a - 2 + h\ge 0$ and 
$a - 2 + h\leq 0$:  
\begin{equation}
\label{eq:f-endpoints-b}
\frac{f(-1+h+a)}{2 r_x} = 
\begin{cases} 
- \frac{1}{r_x} + h \left(1 + \frac{1}{r_x}\right) 
& \mbox{ if $a -  2 + h \ge 0$}, \\
\frac{1}{r_x} & \mbox{ if $a - 2 + h \leq 0$ }
\end{cases} 
\end{equation}
\end{subequations} 
{}From $h \in (0,1]$ and (\ref{eq:f-endpoints-a}), (\ref{eq:f-endpoints-b}) we obtain 
\begin{equation}
\label{eq:f-endpoints-max-left}
\max\{ f(-1+h-a),f(-1+h+a)\}  = f(-1+h-a). 
\end{equation}

We are now ready for a further analysis, for which 
we distinguish the cases that $s_1$ is convex or concave. 
Indeed, only these two cases can occur since the function $s_1$ is the square root of a polynomial of degree $2$, 
and a calculation shows that 
$$
\frac{d^2}{dz^2} \sqrt{\alpha z^2 + \beta z + \gamma} = \frac{4 \alpha \gamma -\beta^2}{4 (\alpha z^2 + \beta z + \gamma)^{3/2}}, 
$$
so that $s_1^{\prime\prime}$ has a sign. We write $s_1(x) = \sqrt{\alpha (x+1-h)^2 + \beta (x+1-h) + \gamma}$ with 
$$
\alpha = 1 - \left(\frac{b}{a}\right)^2, \quad \beta = -2 (2-h), \quad \gamma = (2-h)^2 + b^2 . 
$$
\emph{The case of $s_1$ concave:} This case is characterized by  $4 \alpha \gamma - \beta^2 \leq 0$, i.e., 
\begin{equation}
\label{eq:lemma:inclusion-101} 
- \left(\frac{b}{a}\right)^2 (2 - h)^2 + b^2 \left(1 - \left(\frac{b}{a}\right)^2\right) \leq 0 
\quad 
\Longleftrightarrow 
\quad 
- (2-h+a)^2 + 2a (2-h+a)  - b^2 \leq 0. 
\end{equation}
Since $s_2$ is affine, the function $f = s_1 + s_2$ is concave. We compute $f^\prime(-1+h-a)$: 
\begin{align*}
f^\prime(-1+h-a) &= s_1^\prime(-1+h-a) + s_2^\prime(-1+h-a) = 
\left(-1 + \frac{b^2/a}{2 + a -h}\right) + \sqrt{1 - b^2/a^2}.  
\end{align*}
In order to see that $f^\prime(-1+h-a) \leq 0$, we observe that this last difference is the sum of two terms
of opposite sign. For $\xi$, $\eta \ge 0$ we have  $-\xi + \eta =(\eta^2 - \xi^2)/(\eta+\xi)$ so that the sign
of $-\xi + \eta$ is the same as the sign of $\eta^2 - \xi^2$. Hence, 
\begin{align*}
\operatorname*{sign} f^\prime(-1+h-a) &= \operatorname*{sign} 
\left[ 
- \left(1 - \frac{b^2/a}{2 + a -h}\right)^2 + (1 - b^2/a^2)
\right] 
\\ & 
= \operatorname*{sign} 
\left[
 \frac{b^2/a^2}{(2+a-h)^2} 
\left( - (2+a-h)^2 + 2 a (2-h+a) - b^2
\right)
\right] 
\stackrel{(\ref{eq:lemma:inclusion-101})}{\leq} 0. 
\end{align*}
We conclude that $f$ has its maximum at the left endpoint $x = -1+h-a$. 
\newline 
\emph{The case of $s_1$ convex:} Since $s_1$ is convex and $s_2$  affine (and thus convex), 
the function  $f$ is convex and therefore attains its maximum at one of the endpoints. We get
\begin{align*}
\sup_{x \in [-1+h-a,-1+h+a]} f(x)  & = \max\{f(-1+h-a),f(-1+h+a)\}  \stackrel{(\ref{eq:f-endpoints-max-left})}{ = }
f(-1+h-a) = 
2 r_x \left( \frac{1}{r_x} + h\left(1 - \frac{1}{r_x}\right)\right). 
\end{align*}

The condition on $\rho_1$ is therefore $\widetilde M \leq \rho_1 + 1/\rho_1$, 
and hence the smallest possible $\rho_1$ is given by the condition 
\begin{equation}
\label{eq:lemma:inclusion-1000}
\frac{\rho_1+1/\rho_1}{2 r_x}  = 
\left( \frac{1}{r_x} + h\left(1 - \frac{1}{r_x}\right)\right).  
\end{equation}
Note that the right-hand side is $< 1$ for every $\rho_0 > 1$ and every $h \in (0,1)$. 
One can solve for $\rho_1$ for given $(\rho_0,h)$, i.e., solve  
the quadratic equation (\ref{eq:lemma:inclusion-1000}) for 
$\rho_1 = \rho_1(\rho_0,h)$. 
We first observe the asymptotic behavior of $\rho_1$: we have 
$\lim_{\rho_0 \rightarrow \infty} \rho_1(\rho_0,h)/\rho_0 = h < 1$. One can check the sightly stronger statement 
that for every $\widehat q \in (\overline{q},1)$ there is 
$\overline{\rho_0} > 1$ 
such that $\rho_1(\rho_0,h)/\rho_0 \leq \widehat q$ for all $h \in [0,\overline{q}] \subset [0,1)$ and 
all $\rho_0 > \overline{\rho_0}$. Hence, we are left with checking the finite range 
$(\rho_0,h) \in [\underline{\rho},\overline{\rho_0}] \times [0,\overline{q}]$. 
For that, we note that function $g: x \mapsto x +1/x$ is strictly monotone increasing. Noting 
$g(\rho_0) = 2 r_x$, the equation (\ref{eq:lemma:inclusion-1000}) takes the form 
$$
g(\rho_1) = 
\underbrace{ \left( \frac{1}{r_x} + h\left(1 - \frac{1}{r_x}\right)\right) 
           }_{< 1}
g(\rho_0), 
$$
from which we get in view of the strict monotonicity of $g$ that 
$\rho_1(\rho_0,h) < \rho_0$ for every $\rho_0 > 1$ and every $h \in [0,1)$. The continuity of the 
mapping $(\rho_0,h)  \mapsto \rho_1$, then implies the desired bound 
\begin{align*}
\sup_{\rho_0 \in [\underline{\rho}, \overline{\rho_0}] \times [0,\overline{q}]} \frac{\rho_1(\rho_0,h)}{\rho_0} < 1.   
\qedhere 
\end{align*}
\end{numberedproof}
%
%
\bibliographystyle{plain} 
\bibliography{hmatrix} \end{document}